\documentclass[12pt]{article}
\usepackage{amssymb,amsmath,amsthm}
\theoremstyle{plain}
\newtheorem{theorem}{Theorem}[section]
\newtheorem{lemma}[theorem]{Lemma}
\newtheorem{corollary}[theorem]{Corollary}
\theoremstyle{definition}
\newtheorem{ex}{Example}

\newtheorem{co}{Comment}
\newtheorem{defn}{Definition}
\newtheorem{notation}{Notation}

\newcommand\dref[1]{Definition~\ref{defn:#1}}
\newcommand\tref[1]{Theorem~\ref{thm:#1}}
\newcommand\lref[1]{Lemma~\ref{lem:#1}}
\newcommand\cref[1]{Corollary~\ref{cor:#1}}
\newcommand\coref[1]{Comment~\ref{co:#1}}

\newcommand\eref[1]{Equation~\eqref{eg:#1}}

\begin{document}
\newcommand\sall{$\sigma$-algebra}
\newcommand\sal{$\sigma$-algebra~}
\newcommand\sals{$\sigma$-algebras~}
\newcommand\ttt{$T,T^{-1}$~transformation~}
\newcommand\tttl{$T,T^{-1}$~transformation}
\newcommand\ttp{$T,T^{-1}$~process~}
\newcommand\ttpl{$T,T^{-1}$~process}
\newcommand\Om{$\Omega$~}
\newcommand\Oml{$\Omega$}
\newcommand\om{$\omega$~}
\newcommand\oml{$\omega$}
\newcommand\omo{$\omega_1$~}
\newcommand\omol{$\omega_1$}
\newcommand\omt{$\omega_2$~}
\newcommand\omtl{$\omega_2$}
\newcommand\omb{$\Omega$~}
\newcommand\ombl{$\Omega$}
\newcommand\ep{$\epsilon$~}
\newcommand\epl{$\epsilon$}
\newcommand\mut{$\mu$~}
\newcommand\mutl{$\mu$}
\newcommand\nut{$\nu$~}
\newcommand\nutl{$\nu$}
\newcommand\de{$\delta$~}
\newcommand\del{$\delta$}
\newcommand\deb{$\Delta$~}
\newcommand\debl{$\Delta$}
\newcommand\al{$\alpha$~}
\newcommand\all{$\alpha$}
\newcommand\be{$\beta$~}
\newcommand\bel{$\beta$}
\newcommand\proa{$...a_{-2}, a_{-1}, a_{0}, a_{1}, a_{2},...$~}
\newcommand\prob{$...b_{-2}, b_{-1}, b_{0}, b_{1}, b_{2},...$~}
\newcommand\proc{$...c_{-2}, c_{-1}, c_{0}, c_{1}, c_{2},...$~}
\newcommand\proX{$...X_{-2}, X_{-1}, X_{0}, X_{1}, X_{2},...$~}
\newcommand\proan{$a_{0}, a_{1}, a_{2},...$~}
\newcommand\probn{$b_{0}, b_{1}, b_{2},...$~}
\newcommand\procn{$c_{0}, c_{1}, c_{2},...$~}
\newcommand\procX{$X_{0}, X_{1}, X_{2},...$~}
\newcommand\proA{$...A_{-2}, A_{-1}, A_{0}, A_{1}, A_{2},...$~}
\newcommand\proB{$...B_{-2}, B_{-1}, B_{0}, B_{1}, B_{2},...$~}
\newcommand\proC{$...C_{-2}, C_{-1}, C_{0}, C_{1}, C_{2},...$~}
\newcommand\proAn{$A_{0}, A_{1}, A_{2},...$~}
\newcommand\proBn{$B_{0}, B_{1}, B_{2},...$~}
\newcommand\proCn{$C_{0}, C_{1}, C_{2},...$~}
\title{Infinite Partitions and Rokhlin Towers}
\author{Steven Kalikow
\thanks{University of Memphis, Department of
Mathematics, 3725 Norriswood, Memphis, TN 38152, USA.}
}

\maketitle
\large {\it This paper is dedicated to the memory of Dan Rudolph.}
\vskip .5 cm
Please see the Dedication to Dan that we included in the paper 
\vskip .5 cm

\noindent $``$Nonintersecting splitting \sals in a non-Bernoulli transformation"
\vskip .5 cm

also in this edition.
\normalsize
\begin{abstract}
We find a countable partition $P$ on\textbf{} a Lebesgue space, labeled $\{1,2,3...$\}, for any non-periodic measure preserving transformation $T$ such that $P$ generates $T$ and for the $T,P$ process, if you see an $n$ on time $-1$ then you only have to look at times $-n,1-n,...-1$ to know the positive integer $i$  to put at time $0$. We alter that proof to extend every non-periodic $T$ to a uniform martingale (i.e. continuous $g$ function) on an infinite alphabet. If $T$ has positive entropy and the weak Pinsker property, this extension can be made to be an isomorphism. We pose remaining questions on uniform martingales. In the process of proving the uniform martingale result we make a complete analysis of Rokhlin towers which is of interest in and of itself. We also give an example that looks something like an i.i.d. process on $\mathbb{Z}^2$ when you read from right to left but where each column determines the next if you read left to right. \end{abstract}
\section{Introduction} The purpose of this paper is to show that the intuition one might have for finite state stationary processes gets destroyed when one considers infinite state stationary processes. The proof of the primary theorem we will use to accomplish this destruction of your intuition can be altered to get a theorem about uniform martingales (also called continuous g functions). This still leaves open questions we have been interested in regarding uniform martingales which we will pose, but we will attempt to suggest the basic technique that can be used to resolve these questions. In the process of getting the uniform martingale result we will need a tool involving Rokhlin towers. Since we need this tool anyway we have decided to clean up the theory of Rokhlin towers in this paper. 
\begin{co}\label{co:infent}You can't accomplish this destruction of your intuition if you insist on looking at partitions with finite entropy. Just about everything which is true of finite partitions is also true of infinite partitions with finite entropy. Hence by necessity all infinite partitions in this paper have infinite entropy. 
\end{co}

Countable and uncountable partitions can be strange. For processes generated by a finite alphabet, one definition of zero entropy is that a process has zero entropy iff the past determines the future. Parry \cite{Parry} Theorem 8.2 page 90 showed that every transformation has a countable generator such that the past determines the future (For the present paper generator means that the whole process (past and future) generates but Parry uses the word generator to mean that the past generates).  Here we improve on that result with \tref{main} below:

\begin{defn} \label{defn:nper} Throughout this paper $T$ is a one to one measure preserving transformation on a Lebesgue space \omb with probability measure $P$ such that $P(\{\omega: T^i(\omega)=\omega$ for any $i\}) = 0$. This is what we mean by a non-periodic transformation.
\end{defn} 

\begin{co}\label{co:tinv} If $T$ is one to one and measure preserving on a Lebesgue space then $T^{-1}$ is also (measurability of $T^{-1}$ is standard theory but not obvious) and if $T$ is non-periodic then $T^{-1} $ is also.
\end{co}

\begin{defn} \label{defn:tp} When $T$ is a measure preserving transformation and $P$ is a partition, the $P,T$ process is the process in which every $\omega \in \Omega$ is assigned to a doubly infinite word $... X_{-2}, X_{-1}, X_0, X_1, X_2,...$  where $X_i$ is the element of $P$ containing $T^i(\omega)$. The word it maps \om to is called its $T,P$ name. The $T,P$ process is the process of $T,P$ names endowed with the measure induced by the map from points to $T,P$ names. Since $T$ is measure preserving it is a stationary process. 
\end{defn} 

\begin{co}\label{co:shift} Let $T$ be a measure preserving transformation $P$ a partition and $... X_{-2}, X_{-1}, X_0, X_1, X_2,...$  be the $P,T$ process, then if the $P,T$ name of \om  is $...a_{-2}, a_{-1}, a_0, a_1, a_2,...$, then the $P,T$ name of $T(\omega)$ is $... b_{-2}, b_{-1}, b_0, b_1, b_2,...$  defined by $b_i :=a_{i+1}$ for all $i$, so the map from \om to its doubly infinite name is a homomorphism onto the doubly infinite sequences of elements of $P$ with induced measure where the transformation on doubly infinite sequences shifts a word to the left. If the $T,P$ process separates points (i.e. if there is a set $Z$ of measure $0$  such that if \omo $\neq$ \omt are not in $Z$ they have different $T,P$ names) then this map is an isomorphism.
\end{co} 

\begin{defn} \label{defn:generates} If the $P,T$ process separates points we say that $P$ generates $T$ which implies that $T$ is isomorphic to the standard shift on the $T,P$ process. In that case by abuse of notation we will say that $T$ is isomorphic to the $T,P$ process.
\end{defn}

Here is the main theorem of the paper.

\begin{theorem}\label{thm:main} Every non-periodic transformation has a countable generator, with pieces labeled $ \{1,2,3,...\}$, such that, letting $h(\omega)$ be the piece of the partition where \om is,  
\\(*) if~ $h(T^{-1}(\omega)) = n$, then
$h(T^{-n}(\omega)), h(T^{1-n}(\omega)), h(T^{2-n}(\omega)),... h(T^{-1}(\omega))$ determines $h(\omega)$.
\end{theorem}

\begin{co} \label{co:main} This says that every non-periodic measure preserving transformation is isomorphic to a stationary process on the integers such that if there is an $n$ in position $-1$ then the $-1,-2,....-n$ terms determine the $0$ term.
\end{co} 
\begin{co}\label{co:zero}\tref{main}, as in virtually all theorems and definitions in ergodic theory, allows an exceptional set of measure 0. From here on we will not mention that a set of measure 0 must be removed when we make a theorem, definition or any other statement. The reader should assume it. 
\end{co} 

We will accomplish \tref{main} in two ways: 1) in such a way that the future says very little about the past and 2) in such a way that the inverse of this process also obeys (*) (which implies that the future determines the past and furthermore does so very finitistically.) In the process of proving this theorem we prove a very strong version of the Rokhlin tower theorem. 
\begin{defn} \label{defn:ro}
$B$ is said to be the base of a Rokhlin tower of height $H$ if $B,T(B),T^2(B)...T^{H-1}(B)$ are disjoint. In that case the collection $\\B,T(B)...T^{H-1}(B)$ is called a Rokhlin tower of size $H$ and the complement of the union of the tower is called the error set of the tower. 

\end{defn}  

\begin{defn} \label{defn:al}
We will call a Rokhlin tower an Alpern tower if for any \om in error set of the Rokhlin  tower,  $T(\omega)$ is in the base of the tower. 
\end{defn}
\begin{co}\label{co:aliff}
($B$ is the base of an Alpern tower of height $H)~~$ iff   
$\\~~(B,T(B)...T^{H-1}(B)$ are disjoint and $B,T(B)...T^{H}(B)$ cover the space)$~~$iff$~~$(it is a Rokhlin tower of height $H$ and it is impossible for both \om and $T(\omega)$ to be in the error set.)  
\end{co}

Alpern towers were generalized to two dimensions in both \cite{Prik} and \cite{Sahin} and to flows by Rudolph \cite{Rud}.

\begin{defn} \label{defn:ser}
A sequence of two or more Rokhlin towers have nested bases (resp. nested error sets) if the base (resp. error set) of each (except the last term of the sequence when the sequence is finite) contains the base (resp. error set) of the next. \end{defn}  

\begin{co}\label{co:neednest}
We will need to use a sequence of Alpern towers in both \tref{main} and a theorem we will develop later in this introduction, \tref{ext} (below) but in the proof of \tref{ext} we will also need them to have nested bases so we will need to establish the following. \end{co}

\begin{lemma}\label{lem:rohlin} 
Let $T$ be a measure preserving non-periodic transformation. For sufficiently rapidly increasing $n_i$, there are sets $B_i$ such that 
\\1)	$B_{i+1} \subset B_i$ for all $i$
\\2)	$B_i,T(B_i),T^2(B_i)...T^{n_i-1}(B_i))$ are disjoint
\\3)	$B_i,T(B_i),T^2(B_i)...T^{n_i}(B_i)$ cover the space
\end{lemma}
\begin{co}\label{co:rap}
When we say that $n_i$ are rapidly increasing we mean more precisely that
\\1)	$n_{i+1} > 2n_i^2 + 2$ for all $i$ and 
\\2) $\sum_{i=1}^{\infty}(n_i^2 /n_{i+1}) <\infty$
\end{co}
\begin{co} This says that for sufficiently rapidly increasing $n_i$ we can get a sequence of Alpern towers of height $n_i$ with nested bases. 
\end{co}

Steve Alpern \cite{Alpern} established (2) and (3) of \lref{rohlin} and the establishment of (1), (2) of \lref{rohlin} and that the union of the sets in (2) cover almost the full space is an old well know result. Here we just combine the techniques of these two proofs together to get all three at once. 

	We will indicate in a comment that by tightening (1) of \coref{rap} we  also get all your dreams to come true; a sequence of Alpern towers with nested bases, nested error sets, and error sets whose size is a rapidly decreasing fraction of the size of the bases. This would appear to completely finish the theory of Rokhlin towers with the ultimate theorem if Lehrer and Weiss \cite{Lehrer} had not opened up a can of worms by asking what happens when you insist that you restrict to a subset of a prechosen set of positive measure. We figured that since we have to prove \lref{rohlin} anyway we might as well complete the theory of Rokhlin towers by finishing what Lehrer and Weiss started (although some future author might decide that this paper does not completely finish the study of Rokhlin towers and might find more to say about them.) However, regarding our study of Rokhlin towers, only \lref{rohlin} is actually used in the rest of this paper and only in the proof of \tref{ext}.
	
	What Lehrer and Weiss show is that under very weak preconditions, for any set of less than full measure there is a Rokhlin tower with any height you wish covering it (except measure 0). In particular you can cover one Rokhlin tower (together with a piece of its error set) with another. This allows you to arrange that a prechosen tower be the first of a sequence of towers with nested error sets and arbitrarily small error sets. In fact since you can choose the height before you choose the set you are covering, not only can the error sets be made to be small but in fact you can make their measure to be a small fraction of the measure of the base of the tower. 
	What if you want the prechosen tower to be the first of two towers with nested bases instead of nested error sets and still get the error set to be small? \lref{rohlin} does not allow you to let the first tower be a prechosen tower. Here the answer is so easy we will present the proof inside the introduction.

\begin{lemma}\label{lem:trivial} Let $T$ be a nonperiodic measure preserving transformation and let $S$ be a set of positive measure which intersects every orbit under $T$ (i.e. for every \om there is a positive $i$ such that $T^i(\omega) \in S$). Then there is a subset of $S$ which is the base of a Rokhlin tower with arbitrarily small (but cannot be forced have measure 0) error set. 
\end{lemma}
\begin{co}\label{always} The precondition is satisfied for any set $S$ of positive measure if $T$ is ergodic.
\end{co}

Proof of \lref{trivial}. Let $f(\omega)$ be the least $i$ such that $T^i(\omega) \in S,~\epsilon >0,~M$ be chosen so large that (the probability that $f(\omega)\geq M) <\epsilon , ~N>M/ \epsilon,~P$ be the partition consisting of $S$ and the complement of $S,~B$ be the base of a Rokhlin tower of height $N$ error set smaller than \ep which is independent of the partition which breaks the space into $N$ names for the $P,T$ process (a standard theorem allows you to choose such a $B$), and then break the the Rokhlin tower into $P$ columns (i.e. call two points in $B$ equivalent if they have the same $P,T$ name of length $N$ and consider the columns above the equivalence classes). A rung of such a column is entirely inside one element of $P$. Consider a column $c$. Call the column good if a point \om in its base obeys $f(\omega)<M$. If $c$ is good let $g(c)$ be the first rung of it in $S$. Then 
$\bigcup\limits_{c~ \text{good}} g(c)$ is the base of a Rohlin tower with height $N-M$ and error set smaller than 3\ep.  \hfill $\Box$
\begin{co}\label{co:notsogood} Letting the $S$ of \lref{trivial} be the base of a prechosen Rokhlin tower gives that prechosen tower to be the first of a sequence of two towers with nested bases such that the error set of the second is as small as you like. This would at first glance make you happy. At second glance you would get upset when you realize that in this proof you have to get the height of the second tower to be big in order to get the error set small. One would want better than that. You would want the size of the error set to be able to be a small fraction of the size of the base. Sorry. We have bad news. 
\end{co} 
\begin{theorem}\label{thm:fraction}
Select a doubly infinite sequence of heads and tails from a fair coin and let $\mu$ be the resulting measure on doubly infinite sequences of heads and tails.  Let $A$ be the event that there is a head at the origin. We will show There does not exist $B \subset A$ such that $B$ is the base of a Rokhlin tower whose error set has measure less than $\mu(B)/6$
\end{theorem} 
\noindent(See acknowledgements)
\vskip .2 cm
What about the possibility that we can get something like \lref{trivial} to hold for Alpern towers? Again bad news. We will show
\begin{align}
& \text{COUNTEREXAMPLE: For every integer $n>3$, there exists a Bernoulli} \notag\\
& \text{transformation $T$ and an Alpern tower of height $n$ whose error set} \notag\\
& \text{has any given size less than $1/(n+1)$ and whose base $BB$ does not} \notag\\
& \text{contain the base of  any Alpern tower of height} > n(n-1) + \lfloor n/2 \rfloor \label{eg:counter}
\end{align} 
\begin{co}\label{co:half}Actually $T$ can be chosen before the $n$ and indeed $T$ can be chosen to be the standard 1/2, 1/2 independent process. The point is that the $T$ we actually use has entropy less then that of the 1/2, 1/2 independent process so we can extend it to the 1/2, 1/2 independent process. This causes our set $BB$ to have more subsets but adding more subsets to $BB$ does not harm the proof that the above counterexample works and nor does it harm the proofs of our corollaries. 
\end{co} 

\begin{corollary}\label{cor:noalpern} Let $T$ be the standard 1/2,1/2 Bernoulli. For every $n>3$ there exists an Alpern tower of height between $n$ and $n(n-1) + \lfloor n/2 \rfloor$ which is not the first of any sequence of two Alpern towers with nested bases except if the second tower is equal to the first up to measure 0.
\end{corollary} 
\begin{proof} Select the Alpern tower of counterexample \eref{counter} and that equation implies $\Theta= \{m: m~$ is the height of an Alpern tower whose base is a subset of $BB\} $(It is not empty because $BB$ itself is such a subset so $n \in \Theta$)  has a maximum $mm \leq n(n-1) + \lfloor n/2 \rfloor$. Select a subset $B$ of $BB$ which is the base of such a tower of height $mm$. We claim that the Alpern tower with base $B$ (Henceforth to be called the first base and first tower) serves as the desired tower. Suppose $B$ is the first of two Rokhlin towers with nested bases, the second having base $BBB$. If the height of the second tower is only 1 it cannot be Alpern because the base is too small. Suppose there is an \om $\in B\setminus BBB$. Since the height of both towers exceeds 1 and since the height of the second tower does not exceed the height of the first tower because $mm$ is the maximum, it is easy to see that both \om and $T(\omega)$ fail to be in the second tower so if $B \setminus BBB$ has positive measure then the second tower is not Alpern.\end{proof}
Going back to arbitrary Rokhlin towers, Since Leher and Weiss get an arbitrary tower to be the first of a sequence of two with nested errors such that the second has small error and \lref{trivial} gets it to be the first of two with nested bases such that the second has small error we can hope that we can get it to be the first with both nested bases and nested errors with small error for the second. Bad news.
\begin{corollary}\label{cor:nestbaserr} Let $T$ be the standard 1/2, 1/2 Bernoulli transformation. There is a Rokhlin tower of arbitrarily large height which is not the first of any sequence of two Rokhlin towers with nested bases and nested error sets unless the second is the same as the first except for measure 0.
\end{corollary} 
\begin{proof} The Alpern tower of \cref{noalpern} serves as our example. In fact that corollary explicitly says \cref{nestbaserr} once one makes the obvious observation that if a sequence of two Rokhlin towers has nested error sets, then if the first is Alpern the second is also Alpern.
\end{proof}

Here are some perverted examples when the alphabet is uncountable. Jonathan King \cite{king} provided a stationary process $ ... X_{-2}, X_{-1}, X_0, X_1, X_2,...$ where each $X_i$ is uniformly distributed on the unit interval and if you list a subsequence $X_{a_{0}}, X_{a_{1}}, X_{a_{2}},...$ in which $a_{i+1} \geq a_i+1$ for every $i$, the subsequence is i.i.d. but where $ X_0$ and $X_1$ determine all $X_i$. Rokhlin \cite{Rokhlin} noticed that any process has an uncountable generator where the conditional entropy of the present given the future is the full entropy of the process (in fact reading backwards in time we get a Markov chain) and yet each term determines the next. This is an easy example so we will present it in this paper. Then in this paper we will introduce a particularly perverted example of this. Here we give an easy example of a $\mathbb{Z}$ action on an uncountable state space which looks somewhat like a $\mathbb{Z}^2$ action on $0$s and $1$s consisting of i.i.d. independent random variables taking one value with probability $1/3$ and the other with probability $2/3$ if you read from right to left but where one column determines the next if you read from left to right.
 
\begin{defn} \label{defn:simiid} We say that a process (in general not stationary) 
\\$... X_{-2}, X_{-1}, X_0, X_1, X_2,...$  is similar to an i.i.d. $1/3, 2/3$  process (in this paper we will simply refer to this as a blue process) if the $X_i$ are independent and each $X_i$ takes on one of the following distributions;
1 with probability 1/3 and 0 with probability 2/3, or 0 with probability 1/3 and 1 with probability 2/3.
\end{defn}

\begin{ex}\label{ex:iid} The example we introduce in this paper is of a set of random variables $X_{i,j}$ taking values in $\{0,1\}$, for all integers $i,j$, such that the columns form a stationary process (a one dimensional stationary process on an uncountable state space) and each column determines the next, but that reading from right to left (i.e. if you invert the $\mathbb{Z}$ action), the process is a Markov chain where conditioned on a given column the next column is similar to an i.i.d. $1/3,2/3$ process.
\end{ex}

This is an easy example so perhaps the reader may want to try to come up with it himself before reading this paper. 

\begin{defn} \label{defn:um} A Uniform Martingale (also called a continuous $g$ function) is a stationary process where the convergence of 

~\\the probability of a given letter occurring at time 0 given the $n$ past 

~\\ to 

~\\the probability of that letter occurring  at time 0 given the entire past 

~\\is actually uniform convergence (i.e. where you only have to know $n$ to know how close you are regardless of what past you are talking about).
\end{defn}
\begin{co}\label{co:ptop}This is equivalent to saying that the function taking the past to the probability of a given letter occurring at time 0 given the past is a continuous function of the past when the past is endowed with product topology. For this reason the phrase ``continuous $g$ function" is used to define the probability of the present given the past.
\end{co}

Uniform Martingales have been studied in many papers including \cite{kal}, \cite{wei}, \cite{um1}, \cite{um2}, \cite{um3}, \cite{um4}. In \cite{kal}, and \cite{wei} it was shown respectively that the \ttt and any zero entropy transformation can be extended to a uniform martingale on a finite state space (where the extension takes the form of the composition of three homomorphisms $T \times B  \rightarrow  U  \rightarrow  S \rightarrow T$ where $T$ is the transformation being extended, $S$ is the uniform martingale and the existence of the homomorphism $S \rightarrow T$ is precisely what is meant by saying that $S$ extends $T$. $B$ is a Bernoulli transformation with arbitrarily entropy (a specific Bernoulli transformation was used but when reading those papers it will be obvious that any Bernoulli transformation could have been used). In both of these papers a similar technique is used and the alphabets considered were finite. The technique involved noting that many pieces of past separately told you the present. We will repeat this technique in this paper. The construction we make for \tref{main} does not cause many pieces of past to tell you the present but we will modify it so that it does. Then we will use the same $T \times B  \rightarrow  U  \rightarrow S \rightarrow T$ proof to extend to a uniform martingale on a countable state space. Actually, although we can carry out such a proof and get $S$ as the extension (and we will) in fact we could use $U$ as our desired extension since the only reason for dropping from $U$ to $S$ in \cite{kal}, and \cite{wei} was to get a finite alphabet but here $S$ ends up with a countably infinite alphabet anyway so we might as well use $U$. Our motivation for dropping to $S$ in this paper is that perhaps this can help the reader to solve some of our open problems. Let us be more specific about what we will prove. It is easy to see that \dref{um} is equivalent to the following for finite alphabet processes.

\begin{defn} \label{defn:umt}
 
 A stationary process is a uniform martingale iff the measure on the present given the $n$ past converges uniformly as $n$ approaches $\infty$ on all pasts in the variation metric.
\end{defn}
However the \dref{um} and \dref{umt} are not equivalent when you pass to a countable alphabet (\dref{umt} is stronger). 
\vskip .5 cm
\noindent {\bf Henceforth for both finite and countably infinite alphabets we use \dref{umt} rather than \dref{um} as the definition of uniform martingale.} 
\vskip .5 cm
\noindent This makes the following theorem as strong as possible.

\begin{theorem}\label{thm:ext} a) Every non-periodic transformation can be extended to a uniform martingale on a countable state space. 

b) In fact, if the transformation can be written as the product of another transformation and a nontrivial Bernoulli process, (e.g. if the transformation has positive entropy and obeys the weak Pinsker condition) then this extension can be made to be an isomorphism.
\end{theorem} 

The reason that we bring up this topic in the current paper is that it is almost an application of \tref{main}; but not precisely. We have to alter  the proof of \tref{main} but we use the same basic technique of proof. However, here we need Alpern towers with nested bases.

Obviously it would be nice if we could reduce this to a uniform martingale on a finite alphabet. Here are questions we hope someone will be able to solve.

~\\Question 1: Can every transformation with finite alphabet be extended to a uniform martingale with finite alphabet? If a countable alphabet is needed, can the measure on the alphabet be made to have finite entropy?

~\\Question 2: Is every positive entropy transformation with finite alphabet isomorphic to a uniform martingale with finite alphabet? (It is easy to see that this is impossible for an aperiodic zero entropy process to be a uniform martingale on a finite alphabet). What about the special case where the transformation can be written as a product of a Bernoulli and another transformation? 

~\\Question 3: Is every transformation isomorphic to a uniform martingale on a countable alphabet even if it does not obey weak Pinsker? (It is not known whether or not there is a transformation that does not obey weak Pinsker)

~\\Question 4: Suppose there is a uniform martingale on a 3 letter alphabet with entropy less than log(2). Is it isomorphic to a uniform martingale on a 2 letter alphabet? Can it at least be extended to a uniform martingale on a 2 letter alphabet?

\section{Acknowledgements}
Our proof of \tref{fraction} was simplified by Paul Balister. It is his simplified proof that we use here. It is not only simpler but also better in the sense that with our original proof we would not have gotten a fraction anywhere near as big as 1/6 in the statement of the result. Karen Johannson noticed that my original \dref{newh} was insufficient.

\section{Rokhlin Towers}
 THIS SECTION CONSISTS OF
 
 \noindent Proof of \lref{rohlin} together with an extension of \lref{rohlin} when \coref{rap} is  tightened

\noindent Proofs of \tref{fraction} and Counterexample \eref{counter}
\vskip .5cm 

xxx   Proof of \lref{rohlin}
\vskip .5cm

Preparation for proof of \lref{rohlin}:
\vskip .4cm
Sublemma: Let $m,n$ be integers $m>0$ and $n\geq m(m-1)$. Then $n$ can be written as a sum of numbers each of which is either $m$ or $m+1$. 
\begin{proof} Write $n$ as $km +r$ where$ ~0\leq r \leq m-1$ and $k\geq(m-1)$. Then $n=(k-r)(m)+r(m+1).$
\end{proof}
\begin{defn} \label{defn:sym}
\deb means symmetric difference.
\end{defn}
\begin{defn} \label{defn:dis}
The distance between two sets $A$ and $B$ is the measure of $A \Delta B$.
\end{defn}
\begin{co}\label{co:met}
It is easily seen that distance defined above defines a complete metric on the class of measurable sets where two sets are regarded to be the same if their distance is 0.
\end{co}
\begin{co}\label{co:cauchy}
If a sequence of sets forms a Cauchy sequence using the above metric, then they converge to a set. However that set is only defined up to measure 0, i.e. if a set is the limit of such a sequence and another set has distance 0 from that set then the latter set is also a limit of the sequence. This is not a problem for us because in ergodic theory, when two sets differ by a set of measure 0 we regard them as the same set. 
\end{co}

Before giving a proof of \lref{rohlin} we first give an idea of the proof. The reader has to read the idea however because it is actually part of the proof. 
\vskip .3 cm
Idea of proof of \lref{rohlin}: We will define all the sets $B_i$ we are looking for as a limit of sets which form a Cauchy sequence in the above metric. Each $B_i$ will be defined in such a way letting $B_{i,k}$ be the $k^{\text{th}}$ approximation of $B_i$  

\noindent (here $B_{i,k}$ will only be defined for $k \geq i$) 

\noindent and insisting that the following two statements hold:
\begin{equation}\label{eg:si}
\left.
\begin{aligned}
&\text{If we fix $k$ and substitute $S_i$ := $B_{i,k}$  for $B_i, ~i\leq k$,}
\\ 
&\text{(1) of \lref{rohlin} will hold  for all $S_i$ for $i < k$, and (2) and (3)}
\\ 
&\text{of \lref{rohlin} and will hold for $S_i$, for all $i \leq k$.}
\end{aligned}
\right\}
\end{equation}

In this paper, we will frequently refer to \eref{si} explicitly with the phrase (referring to \eref{si}).

\begin{equation}\label{eg:summable}
\begin{aligned}
&\text{For fixed $i$, the distances between $B_{i,k}$ and $B_{i,k+1}$ will be summable.}
\\ 
&\text{(This condition forces Cauchy).}
\end{aligned}
\end{equation}

Once we define the $B_{i,j}$ and establish \eref{si} and \eref{summable} it follows that as $k$ approaches infinity, $B_{i,k}$ converges to a set $B_i$ and \lref{rohlin} will hold. 

\vskip.5cm

Proof of \lref{rohlin}:

\begin{defn} \label{defn:gomS}
 For any \om $\in$ \omb and any $S \subset$ \omb, $J(\omega,S)$ is the least positive integer such that $T^{ J(\omega,S)} (\omega) \in S$.\end{defn}
 \begin{co}\label{infinite} Theoretically, if $T$ is not ergodic, $J(\omega,S)$ could be infinite (i.e. if $T^i(\omega)$ may fail to be in $S$ for any $i$) but in this proof we will be careful to use sets $S$ such that $J(\omega, S)$ is infinite only on a set of measure $0$.
 \end{co}
 \begin{co}\label{co:meaningg}
Note that (2) and (3) of \lref{rohlin} are equivalent to saying that $J(\omega,B_i)$ is either $n_i$ or $n_i +1$ for any $\omega \in B_i$. \end{co}

[Definition of $s$: 
It is known that, given that $T$ is aperiodic, one can select a set $s$ such that the sets
\begin{equation}\label{eg:disjoints} \text{$T^i(s)$ are disjoint for $0 \leq i < n_1(n_1-1)$}
\end{equation}
and 
\begin{equation}\label{eg:covers} \bigcup_{i=0}^{\infty} T^i(s)=\Omega
\end{equation}
which implies $J(\omega, s)$ is infinite only on a set of measure 0.
There would exist such an $s$ no matter what positive integer is used instead of $n_1(n_1-1)$. \eref{covers} (no matter what small set $s$ we use) is immediate if $T$ is ergodic and to get such an $s$ when $T$ is not ergodic takes a little work but is not hard and is known.] 
\vskip .5cm
Construction of $B_{1,1}$:
\vskip .5cm
We can partition $s$ into sets $s_i$ where \om $\in s_i$ iff $\omega \in S$ and $J(\omega,s) = i$. By \eref{disjoints}, $T^i(s)$, $0<i<n_1(n_1-1)$, is disjoint from $T^0(s)=s$ ($T^0$ is the identity map for any transformation $T$), and thus $s_i =\emptyset$ when $i<n_1(n_1-1)$. By the sublemma start with a nonempty $s_i$ and write $i$ as a sum of some numbers (say $j_i$ numbers) $a_0, a_1, a_2,... a_{j_i}$ each of which is either $n_1$ or $n_1+1$. We then let 
\\
\\
$B_{1,1}=\bigcup\limits_{\{i:s_i \neq \emptyset\},~0 \leq q < j_i}T^{\sum_{k=1}^{q}a_k}(s_i) $, 
\vskip.5cm
\noindent and (2) and (3) of \lref{rohlin} hold (referring to \eref{si}) for $B_{1,1}$. Note that when $q=0, \sum_{k=1}^{q} a_k=0$ so $S \subset B_{1,1}$. 
\vskip.5cm
Now fix an $I$ and we assume we have defined $B_{i,j}$ for all $i\leq j \leq I$ obeying (2) and (3) of \lref{rohlin} and (1) of \lref{rohlin} for $i<j$ (referring to \eref{si}). We wish to define $B_{j,I+1}$ for all $j \leq I+1$ obeying (2) and (3) and (1) of \lref{rohlin} for $j<I+1$(referring to \eref{si}). 
\vskip .5cm
Construction of $B_{I+1,I+1}$:
\vskip .5cm
Construct $B_{I+1,I+1}$ precisely the same way that we constructed $B_{1,1}$ replacing $n_1$ with $n_{I+1}$ throughout the construction (this involves redefining $s$ accordingly, i.e. using a different $s$).
\vskip .5cm
Construction of $B_{j, I+1}$ for $j < I+1$:
\vskip .5cm
Our strategy will be to first define $B_{I,I+1}$ then $B_{I-1,I+1}$ then, $B_ {I-2,I+1}$ etc. until we define $B_{1, I+1}$. The strategy for getting from $B_ {i, I+1}$ to $B_ {i-1,I+1}$ will be the same for all $i$ except $i=I+1$ so we will first show how to get from $B_{I+1,I+1}$ to $B_{I,I+1}$, then more quickly indicate how to get from $B_{I,I+1}$ to $B_{I-1,I+1}$ and then even more quickly indicate how to accomplish the general case. The increase in the brevity of our explanation at each stage is justified by the assumption that as we go from one case to the next we expect the reader to see the pattern.
\newline \indent Fix an $\omega \in B_{I+1,I+1}$. We are now about to define values $a(\omega)$ and $b(\omega)$ which will both exist and satisfy $a(\omega) < b(\omega)$. Recall that since $\omega \in B_{I+1,I+1}$, it follows that $J(\omega,B_{I+1,I+1})$ is either $n_{I+1}$ or $ n_{I+1}+1$.  \begin{defn} \label{defn:aombom}
 
Let $a(\omega)$ be the least integer greater than $(n_I)(n_I-1)$ such that $T^{a(\omega)}(\omega) \in B_{I,I}$ and let $b(\omega)$ be the greatest integer less than  $J(\omega,B_{I+1,I+1}) - (n_I)(n_I-1)$ such that $T^{b(\omega)}(\omega)\in B_{I,I}$. 
\end{defn} 
We wish to emphasize that $\omega \in B_{I+1,I+1}$ and that for such \om 
\begin{equation}\label{eg:BII} T^{a(\omega)}(\omega) \in B_{I,I} \text{ and }T^{b(\omega)}(\omega) \in B_{I,I}
\end{equation}

Since, by induction, (3) of \lref{rohlin} holds for $B_{I,I}$,(referring to \eref{si}) it follows that for any \omol, $J(\omega_1,B_{I,I})\leq n_I+1$. For $\omega \in B_{I+1,\hskip .1cm I+1}$, applying that to 
$\omega_1 = T^{n_I(n_I-1)}(\omega)$ we have that 
\begin{equation}\label{eg:abound} 
a(\omega) \leq n_I(n_I-1) +n_I +1 = n_I^2 + 1\end{equation}  
and applying it to $\omega_1 = T^{n_{\raise-2pt\hbox{$\scriptscriptstyle I+1$}}-n_{\raise-2pt\hbox{$\scriptscriptstyle I$}}(n_{\raise-2pt\hbox{$\scriptscriptstyle I$}} - 1)-n_{\raise-2pt\hbox{$\scriptscriptstyle I$}} - 1}(\omega)$ 
we have that
\begin{equation}\label{eg:bbound} 
b(\omega)\geq n_{I+1}  - n_I(n_I-1) - n_I -1 = n_{I+1} - n_I^2 -1\end{equation} 
so by (1) of \coref{rap}, $a(\omega)<b(\omega)$. Keeping in mind that $T^0$ is the identity transformation for any transformation $T$, 

\begin{defn} \label{defn:changeset}

let $changeset = \bigcup\limits_{\omega \in B_{I+1,I+1}}\{T^i(\omega):0 \leq i\leq a(\omega)~ \text{or}~b(\omega)\leq i <J(\omega,B_{I+1,I+1})\}$.
\end{defn}
The ``0" in the expression $0\leq i\leq a(\omega)$ in the definition of $changeset$ establishes that $B_{I+1,I+1} \subset changeset$.
Let $sameset$ be the complement of $changeset$.

\vskip .5cm

For all $j\leq I$ we will construct $B_{j,I+1}$ in such a way that $B_{j,I+1} \cap sameset$ = $B_{j,I} \cap sameset$. This establishes \eref{summable} because it means that the measure of $B_{j,I+1} \Delta B_{j,I}$ is bounded above by the measure of $changset$ which (by i and ii below) is bounded above by $2[n_I^2+1]/n_{I+1}$ so \eref{summable} follows from (2) of \coref{rap}. 

\vskip .5cm

i) By the definition of $changeset$, \eref{abound} and \eref{bbound}, 
\newline $changeset \subset [\bigcup\limits_{i=0}^{n_I^2+1}T^i(B_{I+1,I+1}]~\cup~[\bigcup\limits_{i=1}^{n_I^2+1}T^{n_{\raise-2pt\hbox{$\scriptscriptstyle I+1$}}-i}(B_{I+1,I+1}))]$

\vskip .5cm

ii) Since $T$ is measure preserving, for all $k$, $T^k(B_{I+1,I+1})$ has the same measure as $B_{I+1,I+1}$ which is less than or equal to $1/n_{I+1}$ because $B_{I+1,I+1}$ satisfies (2) of \lref{rohlin} (referring to \eref{si}).

\begin{co}\label{co:changesetv} The right way to visualize $changeset$ is to think of it as a union of orbit intervals from $T^{b(\omega_1)}(\omega_1)$ to $T^{a(\omega_2)}(\omega_2)$ where $\omega_1 \in B_{I+1,I+1}$ and
\begin{align}
&  \omega_2 = T^{J(\omega_1,B_{I+1,I+1})}(\omega_1), \text{ i.e.,the next element of the orbit of \omo after \omo}\notag \\
&\text{which is in } B_{I+1,I+1}.\notag
\end{align}
Throughout this section the way we will change $B_{j,I}$ to get $B_{j,I+1}$ will be to change it on such intervals so that $B_{j,I}$ ``matches up" with $B_{j,I+1}$ on both the $T^{b(\omega_1)}(\omega_1)$ and $T^{a(\omega_2)}(\omega_2)$.
\end{co} 

\vskip .5cm

Construction of $B_{I,I+1}$:

\vskip .5cm

We want to alter $B_{I,I}$ on $changeset$ to get $B_{I,I+1}$ so that (2) and (3) of \lref{rohlin} continues to be true for $B_{I,I+1}$ (referring to \eref{si}) and furthermore so that $B_{I+1,I+1}$ becomes a subset of $B_{I+1,I}$.

By the sublemma, using the same technique we used to construct $B_{1,1}$ and $B_{I+1,I+1}$, for every $\omega \in B_{I+1,I+1}$, we can produce a finite set of points in $changeset$

\vskip.5cm

$S_1(\omega) \subset   \{T^i(\omega):0\leq i\leq a(\omega)\}$, 
\vskip.5cm
listing the terms of $S_1(\omega)$ as
\vskip.5cm
$S_1(\omega) =\{d_0, d_1, d_2,... d_k\}$ 

\vskip.5cm

and we can produce a finite set of points in $changeset$

\vskip.5cm

$S_2(\omega) \subset \{ T^i(\omega):b(\omega)\leq	 i<J(\omega, B_{I+1,I+1},)\},$
\vskip.5cm
listing the terms of $S_2(\omega)$ as
\vskip.5cm
$S_2(\omega) =\{e_0, e_1, e_2,... e_n\}$ 

\vskip.5cm
in such a way that 

\vskip.5cm

$d_0=\omega$, either $T^{n_{\raise-2pt\hbox{$\scriptscriptstyle I$}}}d_i= d_{i+1}$ or $T^{n_{\raise-2pt\hbox{$\scriptscriptstyle I$}}+1}d_i = d_{i+1}$ for $i<k;~ d_k= T^{a(\omega)} (\omega)$,

\vskip.5cm

$e_0 = T^{b(\omega)} (\omega)$, either $T^{n_{\raise-2pt\hbox{$\scriptscriptstyle I$}}}e_i = e_{i+1}$ or $T^{n_{\raise-2pt\hbox{$\scriptscriptstyle I$}}+1}e_i= e_{i+1}$ for $i<n$, and either 
\indent $T^{n_{\raise-2pt\hbox{$\scriptscriptstyle I$}}}e_n = T^{J(\omega,B_{I+1,I+1})} (\omega)$ or $T^{n_{\raise-2pt\hbox{$\scriptscriptstyle I$}}+1}e_n = T^{J(\omega, B_{I+1,I+1})}(\omega)$.

\vskip.5cm

Let $B_{I,I+1}= [B_{I,I}\cap sameset] \cup [\bigcup\limits_{\omega \in B_{I+1,~I+1}}( S_1(\omega) \cup S_2(\omega))].\\~~~~$ Note that for $\omega \in B_{I+1,I+1}$,

\begin{equation}\label{eg:inducab} T^{a(\omega)}(\omega) \in B_{I,I+1} \text{ and }T^{b(\omega)}(\omega) \in B_{I,I+1}
\end{equation}

For $ \omega \in B_{I+1,I+1}$, the fact that we insist that $d_0=\omega$ in our definition of $S_1(\omega)$ establishes that $B_{I+1,I+1} \subset B_{I,I+1}$. It is easy to see (2) and (3) of  \lref{rohlin} (referring to \eref{si}) for $B_{I,I+1}$ if you keep in mind that $T^{a(\omega)}(\omega)$ and $T^{b(\omega)}(\omega)$ are both in $B_{I,I} \cap B_{I,I+1}$, (by \eref{BII} and \eref{inducab}), that $ B_{I,I+1} \cap sameset = B_{I,I} \cap sameset$, and that (2) and (3) of \lref{rohlin} already works for $B_{I,I}$.

\begin{co}\label{co:skip}We just established $B_{I,I+1}$ by appropriately breaking up big intervals in $changeset$ into smaller intervals of size $n_I$ and $n_I+1$ appropriately. Now we are going to appropriately break those subintervals into subintervals of size $n_{I-1}$ and $n_{I-1}+1$ and those into subintervals of size $n_{I-2}$ and $n_{I-2}+1$ etc. Readers may feel at this point that the rest of the proof is clear and that they can skip the rest of the proof of this lemma. In that case, skip to the next bold face sentence.
\end{co}  
Construction of $B_{I-1,I+1}$:
\vskip .5cm
Let $\Theta_1 = \{\omega$: there exists $\omega_1 \in B_{I+1,I+1}$ such that $T^{a(\omega_1)}(\omega_1)=\omega\}$. Let $\Theta_2= changeset \setminus \Theta_1$ so that $changeset = \Theta_1 \dot \cup \Theta_2$

\begin{co}\label{co:JBI1} If $\omega_1 \in B_{I+1,I+1},~ \omega_2 \in B_{I+1,I+1},~ 1\leq k_1\leq n_{I+1}, 1\leq k_2\leq n_{I+1}$, and if $T^{k_1} \omega_1 = T^{k_2} \omega_2$, then $k_1=k_2$ and $\omega_1=\omega_2$ because otherwise if we just have $k_1=k_2$ we contradict injectivity of $T$ and if $k_1 \neq k_2$ then $T^{k_1} B_{I+1,I+1}$ intersects  $T^{k_2} B_{I+1,I+1}$ which means (by injectivity ) that $T^{k_1-1}B_{I+1,I+1}$ intersects $T^{k_2-1}B_{I+1,I+1}$ contradicting (2) of \lref{rohlin} (referring to \eref{si}). This shows uniqueness of \omo in the definition of $\Theta_1$ and avoids ambiguity in the proof of the following claim.
\end{co}

\vskip .5cm

Claim: 
For all $\omega \in B_{I,I+1} \cap \Theta_2$, every $T^i(\omega)$ where $0 \leq i< J(\omega,B_{I,I+1})$ is also in $\Theta_2$.

\begin{proof}

We have to analyze our definitions of $\Theta_2,~ a(\omega)$ and $b(\omega)$. Since $\omega \in \Theta_2,   ~\omega = T^j(\omega_1)$ for $ \omega_1 \in B_{I+1,I+1}$ where $0 \leq j<a(\omega_1)$ or $b(\omega_1) \leq j< J(\omega_1,B_{I+1,I+1})$. Select $i$ with $0 \leq i< J(\omega,B_{I,I+1})$.
 
Case 1: $0 \leq j< a(\omega_1)$:  By time $i, 0 \leq i< J(\omega,B_{I,I+1}),~ T^i(\omega)$ has not yet returned to $B_{I,I+1}$ so since $T^i(\omega)=T^{i+j}(\omega_1)$, by \eref{inducab}, $j+i < a(\omega_1)$ and thus $T^i(\omega) =T^{j+i}(\omega_1) \in \Theta_2$. 

Case 2: $b(\omega_1)\leq j<J(\omega_1,B_{I+1,I+1})$: By time $i, 0 \leq i< J(\omega,B_{I,I+1}), T^i(\omega) =T^{j+i}(\omega_1)$ is not yet in $B_{I,I+1} \supset B_{I+1,I+1}$ so 
$i+j < J(\omega_1,B_{I+1,I+1})$. \end{proof}

We will alter $B_{I-1,I}$ to get $B_{I-1,I+1}$ on $\Theta_1 \cup \Theta_3$ where $\Theta_3=\{T^i(\omega):\omega \in B_{I,I+1} \cap \Theta_2$ and $0 \leq i < J(\omega, B_{I,I+1})\}$. The claim guarantees that this won't effect $sameset$. Actually $\Theta_3 = \Theta_2$ because using a proof similar to the proof of the claim, if we start on a point in $\Theta_2 \setminus B_{I,I+1}$ and read backwards in its orbit until we reach $B_{I,I+1}$ we will not leave $\Theta_2$. Select $\omega \in B_{I,I+1} \cap \Theta_2$. Since $n_{I}>n_{I-1}(n_{I-1}-1)$ select finite set $S(\omega) \subset \{T^i(\omega) : 0 \leq i \leq J(\omega,B_{I,I+1})\}, S(\omega) = \{d_0, d_1, d_2,... d_k\}$ such that $d_0 = \omega$,  either $T^{n_{\raise-2pt\hbox{$\scriptscriptstyle I-1$}}}d_i = d_{i+1}$ or $ T^{n_{\raise-2pt\hbox{$\scriptscriptstyle I-1$}}+1}d_i = d_{i+1}$ for $i<k$. If $T^{ J(\omega,B_{I,I+1})}(\omega) \in \Theta_1$, arrange that $d_k= T^{ J(\omega,B_{I,I+1})}(\omega)$. Otherwise arrange that
either $T^{n_{\raise-2pt\hbox{$\scriptscriptstyle I-1$}}}d_k = T^{ J(\omega,B_{I,I+1})}(\omega)$ or $T^{n_{\raise-2pt\hbox{$\scriptscriptstyle I-1$}}+1}d_k = T^{ J(\omega,B_{I,I+1})}(\omega)$ and then let 

$B_{I-1,I+1} =(B_{I-1,I}\cap sameset) \cup \bigcup\limits_{ \omega \in B_{I,I+1} \cap changeset} S(\omega)$

For induction purposes we note that 
\begin{equation}\label{eg:inductwo} \text{For $\omega_1 \in B_{I+1,I+1}$, both $T^{a(\omega_1)}(\omega_1)$ and $T^{b(\omega_1)}(\omega_1)$ are in $B_{I-1,I+1}$}
\end{equation}
because 
\begin{align}
& T^{a(\omega_1)}(\omega_1)~\text{is the } d_k \text{ where \om is the last element}\notag\\ 
& \text{of the orbit of \omo before }T^{a(\omega_1)}(\omega_1) \text{ which is in } B_{I,I+1}\notag
\end{align}
and
\begin{align}
&T^{b(\omega_1)}(\omega_1) \text{ is the } d_0 \text{ where } \omega = T^{b(\omega_1)}(\omega_1) \text{ which}\notag\\
& \text{works because } T^{b(\omega_1)}(\omega_1) \in B_{I,I+1} \text{ by \eref{inducab}.} \notag
\end{align}

Since both $T^{a(\omega_1}(\omega_1)$ and $T^{b(\omega_1}(\omega_1)$ are in $B_{I,I} \cap B_{I-1,I+1} \subset B_{I-1,I} \cap B_{I-1,I+1}$ from (\eref{BII}, \eref{inductwo} and that $B_{I,I} \subset B_{I-1,I}$ is part of our induction hypotheses), the way $B_{I-1,I+1} \cap changeset$ was just constructed together with the fact that \eref{si} works for $B_{I-1,I}$ implies that \eref{si} works for $B_{I-1,I+1}$.
 \vskip .5cm
Construction of $B_{j,I+1}, j<I-1$:
\vskip .5cm
As before for all $\omega \in B_{j+1,I+1} \cap \Theta_2$, all $T^i(\omega)$ where $0 \leq i< J(\omega,B_{j+1,I+1})$ are also in $\Theta_2$. Proceed as above.\hfill $\Box$
\vskip .5cm
{\bf This is all you need of this section to follow the rest of the paper. You can safetly skip the rest of this section if you wish.}
\vskip .5cm
xxx Extending \lref{rohlin}:
\vskip .5 cm
Before we extend the lemma we must first extend the sublemma. Fix an $n$. Let $N>>n$. Subtract $n+1$ from $N$ and then divide by $n$ so that the integer value of the quotient is $Q=\lfloor (N-n-1)/n \rfloor$ with a remainder of $r$ and you have $N=n+1 + Qn +r=n+1 +(Q-r)n +r(n+1)$ where $0 \leq r < n$. This expresses $N$ as a sum of $``n"$s and $``n+1"$s where the number of $``n"$s in the expression is $Q-r>Q-n$ and the number of $``n+1"$s is $r+1 \leq n$. We have established
\vskip .5 cm
Extended sublemma: Fix an $n, ~~N,$ and $R>0$. If $N$ is so large that $(\lfloor (N-n-1)/n \rfloor -n)/n>R$ (e.g. if $N>Rn^2+n^2 + 2n + 1$) then $N$ can be written as a sum of terms each of which is either $n$ or $n+1$ such that 
\vskip .5cm
(the number of $``n"$s) / (the number of $``n+1"$s) 
\begin{equation}\label{eg:R}
\noindent \text{is at least $R$ but such that there is at least one}  ``n+1".\end{equation} 

It is time for us to pay attention to the proof we just concluded to see exactly when we used the sublemma and what happens when we replace it with the extended sublemma. The reason we use the sublemma over and over in this proof is that we are trying to establish Alpern towers. Adding $n$ means that you are on the base of one of these towers and that you have to wait time $n$ until you reach the base again. Adding $n+1$ means you have to wait until time $n+1$ before reentering the base. In this way we obtain an Alpern tower of height $n$. When you add $n+1$ that means that when you leave the base, the last before reentering it you will be in the error set. To get nested error sets you want to make sure that

\noindent {\bf Every time in the proof that you add a bunch of $``n"$s and $``n+1"$s to get a large number the last term you add is a $``n+1"$.}

That way when considering two successive bases $B_1$ and $B_2$ (where $B_2$ is the smaller one) you can be sure that after leaving $B_2$ the last time before you reenter $B_2$ you are in the error set of $B_1$. This gives nested error sets. The extended sublemma promises that at least one of the terms being added is an $``n+1"$ so we can make sure the last term added is an $``n+1"$. 
The only time you see error at all is when you add $``n+1"$ whereas every time you see either $``n"$ or $``n+1"$ you enter the base once. Thus \vskip.5cm
\noindent(the ratio of the measure of the error to the measure of the base)$~=\#(``n+1"s)/(\text{total number of} ``n"\text{s and} ``n+1"s)$ which by \eref{R} is bounded above by $1/(R+1)$. 
\vskip.5cm
But to apply the extended sublemma we first need the preconditions of it to hold. Just make the necessary changes to make that possible. First, the expression $n_I(n_I-1)$ in the definitions of $a(\omega)$ and $b(\omega)$ (\dref{aombom}) has to be replaced by $R_I n_I^2+n_I^2 + 2n_I+1$ where $1/(R_I+1)$ is the desired upper bound for the ratio of the measure of the error set of the $``I"$th tower to the measure of the base of the $``I"$th tower. To get $b(\omega)>a(\omega)$ you will have to tighten (1) of \coref{rap} accordingly. We summarize all this as follows.

\begin{co}\label{co:dreams} If the heights of the towers are allowed to increase sufficiently rapidly, we can get a sequence of Alpern towers which have nested bases, nested error sets and error sets decreasing as rapidly as you might want even in comparison to the bases but the ratios of the size of the error sets to the size of the bases that you want determines how fast your towers must grow. 
\end{co}   
xxx Proof of \tref{fraction}
\begin{proof}
We let $T$ be the shift to the right and assume the existence of $B$. Randomly select \om in the space, select a huge $N$ and let

\noindent$S= \{T(\omega),T^2(\omega)...T^N(\omega)\}$, and let

\vskip .2cm 
  
$Ev = $``the number of elements of $S$ in the error set is less than $N\mu(B)/5$". 

\vskip .2cm 

\noindent We will derive a contradiction by showing that the probability of $Ev$ approaches 0 as N approaches $\infty$. Assume $Ev$. Let $n$ be the number of elements of $S$ in the error set. Then 

1)      $n< N\mu(B)/5$

Let $H$ be the height of the tower. Since $B$ is the base of the tower it follows that 

2)      $\mu(B) \leq 1/H$

Now list the elements of $S \cap B$ in order as 
$S1=\{\omega_1, \omega_2, \omega_3...\omega_k\}$. This defines $k$. Obviously nothing in $S1$ is in the error set of the tower since $B$ is the base of the tower. There are $k-1$ towers completely in $S$ and perhaps a piece of the tower preceding those towers and a piece of the tower succeeding those towers (This is abuse of notation. When we say that $``$There are $k-1$ towers" we mean $``$you pass through the tower $k-1$ times").
It follows from the fact that each tower has height $H$ that 

3)      $(k-1)(H) +n <N < (k+1)H+n.$ 

First let us choose the number of terms in $S$ before \omo which are not in the error set. Since the height of the tower is $H$ the number of such terms is less than $H$ so we can choose this number to be either $0,1,2,...$or $H-1$

Choice 1: The number of ways we can choose the number of nonerror terms before \omo in $S$ is $H$.

Now we are about to choose $k+1$ numbers $a_1, a_2, ...a_{k+1}$.
Regarding the elements of $S1$ as being a subsequence of the elements of $S$, 
\vskip .5 cm

\noindent $a_1$ is the number of elements of $S$ in the error set listed before \omol.

\vskip 0 cm

\noindent $a_2$ is the number of elements of $S$ in the error set between \omo and 
$\omega_2$.

\vskip 0 cm

\noindent $a_3$ is the number of elements of $S$ in the error set between $\omega_2$ and $\omega_3$.

\vskip .2cm

\noindent .

\vskip .2cm

\noindent .

\vskip .2cm

\noindent .

\vskip .2cm

\noindent $a_k$ is the number of elements of $S$ in the error set between $\omega_ {k-1}$ and $\omega_ {k}$.

\vskip 0cm

\noindent $a_{k+1}$ is the number of elements of $S$ in the error set after $\omega_ {k}$.

\vskip .5cm
These $k+1$ nonnegative integers have to add to $n$. If you add one to each of them you get $k+1$ positive numbers that add to $n +k +1$.
There is a standard trick which shows that the number of ways to pick such a sequence of $n+k+1$ positive integers is ${n+k\choose k}$.
\vskip .5cm
Choice 2: The number of ways to choose the $a_i$ sequence is ${n+k \choose k} = 
{n+k \choose n}$
\vskip .5cm
Once we have made the first two choices, the values of $S1$ as elements of $S$ are determined (e.g. if $k>21$ you know for which value of $j$ we have that $\omega_{21} = T^j(\omega))$. This selects out $k$ explicit terms in $S$ which are in $S1$ and since $S1 \subset B \subset A$ we get $k$ explicit terms in $S$ which are in $A$. For any $k$ such terms the probability that they are all in $A$ is $1/2^k$.
\vskip .5cm

We have established that the probability of $Ev$ is bounded above by 
\vskip .5 cm
$x:= H  {n+k \choose  n}  1/2^k$
\vskip .5 cm
\noindent and all that remains is to show that $x$ is tiny. We will make use of approximations to make our arguments easier and cleaner but we are sure the reader will agree that the inequalities and the speed at which we prove $x$ to go to zero overpowers the errors in these approximations. If the reader thinks we are cutting it close we could have made it blatantly obvious if we had used 10 instead of 6 in the statement of \tref{fraction}  allowing us to use 9 instead of 5 in the definition of $Ev$ and that would have been good enough to show that we have a counterexample.  
For your convenience we will repeat the equations we have established.

\vskip .5 cm
      1)      $n< N\mu(B)/5$
      \vskip 0 cm
      2)  $\mu(B)<1/H$
      \vskip 0 cm
      
      3)  $(k-1)(H) +n <N < (k+1)H+n$.
      \vskip 0 cm

\vskip .5 cm
      
Let $M=N/H$. $\approx$ means approximately equal.

\noindent From (1) and (2)

\noindent 4)      $n<M/5$

\noindent from 3,
 
\noindent 5)      $k \approx (N-n)/H$ which by (4) is essentially $M$

Now we analyze the $H$, the ${n+k \choose n}$ and the $1/2^k$ in the definition of $x$.
\vskip .5 cm
First term= $H$
\vskip .5 cm
Second term = ${n+k \choose n }<{M/5+k \choose M/5}\approx {M  \choose  M/5}<(M)^{M/5}/(M/5)!
 \approx (5 e)^{M/5}$
\vskip .5 cm
Third term = $1/2^k \approx 1/2^M$
\vskip .5 cm
Multiplying all this together gives $H ((5e)/32)^{M/5}$ which goes rapidly to 0 as $M$ goes to $\infty$.
\end{proof}
xxx Proof of \eref{counter}
\begin{proof}  Let \ep be the desired size of your error set. Then
\begin{equation}\label{eg:select}
~\epsilon<1/(n+1)
\end{equation}
Consider the following two words in ``0"s and ``1"s:

\vskip .5  cm

a ``0" followed by $n-1~~~~``1"$s \hskip .2cm which we will refer to as a $n$ block and

\vskip .01 cm

a ``0" followed by $n~~~~``1"$s  \hskip .2cm  which we will refer to as a $n+1$ block. 

We will independently concatenate $n$ blocks and $n+1$ blocks to get an infinite word where each time we select an $n+1$ block with probability $\epsilon n/(1-\epsilon)$ which is less than 1 by \eref{select}. This gives a generic word for an $n+1$ step  mixing Markov process and such processes are know to be Bernoulli. We use that process, let $T$ be the shift of a doubly infinite word on that process and $BB$ be the event that  there is a ``0" at the origin. It is easily seen that $BB$ is the base of an Alpern tower of height $n$ and error set of size \epl. Now select $N>n(n-1)+\lfloor n/2 \rfloor$. Divide $N-\lfloor n/2 \rfloor$ by $n$ to get an integer part $Q$ with a  remainder of $r$ and that enables us to write $N$ as $N=\lfloor n/2 \rfloor +Qn +r$ where $Q \geq n-1$ and $r\leq n-1$

Rewrite that as
\begin{equation}\label{eg:N}
N= \lfloor n/2 \rfloor +r(n+1) +(Q-r)n~\end{equation} 
and note that $Q-r \geq 0$. For any point (i.e. doubly infinite word) \om we say that \om is in the beginning of its $BB$ block if it is in $BB$, i.e. if it has a ``0" at the origin. Points near the middle of a $BB$ block obviously have a ``1" at the origin. Although $n+1$ blocks may be rare it is nonetheless true that any finite sequence of blocks occur with positive probability and hence will eventually occur. In particular the following will eventually occur in the output of a randomly chosen \om. 
You will see
\vskip .5cm
$r$ blocks which are $n+1$ blocks which we will call $D$ blocks
\vskip .5cm
then
\vskip .5cm
$Q-r$ blocks which are $n$ blocks which we will call $E$ blocks 
 \vskip .5cm
then
\vskip .5cm
$r$ blocks which are $n+1$ blocks which we will call $F$ blocks
\vskip .5cm
then 
\vskip .5cm
$Q-r$ blocks which are $n$ blocks which we will call $G$ blocks
\vskip .5cm

Suppose for a contradiction that there is an Alpert tower of size $N$ whose base $CC$ is a subset of $BB$. Let \omo be the translate of \om whose origin is the 1 at the end of the last $E$ block (or the last $D$ block if $Q-R = 0$). Look where \omo is in it's the $N$ tower and read backwards until you get to the bottom of the $N$ tower to get a point $x$ which is in $CC$. Then read forward until you get to the next time that you are in $CC$ at a point $y$. The distance you have to travel in the orbit to get from $x$ to $y$ is either $N$ or $N+1$. Since $x$ and $y$  are in $C \subset BB$, the both have a 0 at the origin. By \eref{N} $x$ is in the beginning of an $D$ or $E$ block (It can't be partway below the first $D$ block or it would have a 1 at the origin). Again by \eref{N} you have to go through exactly $r$ blocks which are $n+1$ blocks to get from $x$ to $y$ and every other block you go through is an $n$ block. But that is impossible because (again by \eref{N}) that would put $y$ near the middle of its $BB$ block causing it to have a 1 at the origin. 
\end{proof}

\section{Proof of \tref{main}}
We assume $T$ to be a non-periodic transformation.

We will first prove \tref{main} in such a way that the future says very little about the past and then do the other extreme; we will modify the proof so that the inverse of the process also obeys the theorem.

\begin{lemma}\label{lem:fg} For any nonperiodic measure preserving transformation $T$ there is a function $f$ from \omb to  $\mathbb{N}$ and a function $g$ from $\mathbb{N}$ to $\mathbb{N}$ such that
\\1)	$|f(T(\omega))- f(\omega)|
\leq 1$ for all \oml. 
\\2)	For all \om and all nonnegative integers $i$, there is a member of 
$\\ \{T(\omega),T^2(\omega) . . . T^{g(i)}(\omega)\}$ where $f$ takes on a value greater than $i$.
\end{lemma} 
\begin{defn} \label{defn:fg} After proving this lemma $f$ and $g$ will henceforth be the functions above given by this lemma except in sections 4 and 6. Sections 4 and 6 are the sections where we discuss uncountable partitions, \tref{main} is not relevant, and $f$ will have a different meaning.
\end{defn} 

\begin{proof}  
 For the purposes of this lemma we don't need nested Alpern towers. We only need a sequence of Alpern towers where the height of the $i^\text{th}$ tower is $n_i$ such that 
\begin{equation}\label{eg:3i} n_i>3i~ \text{for all}~ i.
\end{equation}
\begin{equation}\label{eg:1ni} \sum\limits_1^\infty (i/n_i)<\infty.
\end{equation}

\noindent Fix $i$  and \oml. Define $j$ to be the number such that \om is in the $j^\text{th}$ rung of the $i^\text{th}$ tower (let it be $\infty$ on the error set). Define a function $f_i$ by 
\begin{equation}\label{eg:fi} 
f_i(\omega)= \begin{cases}              
i ~~~~~~~~~~~~\text{ if } 1 \leq j \leq i\\
2i-j~~~~~~ \text{if } i+1\leq j \leq 2i \\ 0~~~~~~~~~~~\text{ otherwise}\end{cases} 
\end{equation}
and it is clear that $f_i$ obeys (1) of \lref{fg}. 

Each rung of the $i^\text{th}$ tower has size at most $1/n_i$ so the support of $f_i$ has measure at most $2i/n_i$. By \eref{1ni}, Borel Cantelli says that (after removing a set of measure 0) every \om  is in only finitely many such supports. That means the following definition makes sense. Let $f$ be the maximum of all $f_i; i$ going from $1$ to $\infty$. It is easy to verify that the maximum of a bunch of functions obeying (1) of \lref{fg} obeys (1) of \lref{fg} so $f$ obeys (1) of \lref{fg}. Let $g(i)= n_{i+1} +1$ . Since the $i+1^\text{th}$ tower is an Alpern tower of size $n_{i+1}$ any stretch of orbit of size $g(i)$ eventually hits the $i+1^\text{th}$ rung of the $i+1$ tower and hence in such a stretch there is a point \om where $f_{i+1}(\omega) = i+1$ and hence $f(\omega)\geq i+1$.
\end{proof}

\vskip .5 cm
Completion of proof of \tref{main}:
\vskip .5 cm
Now let $P_i$ be an increasing sequence of finite partitions which separates points and label each piece of each $P_i$ with a distinct positive integer i.e. for any $i,j$,  a piece of $P_i$ and a piece of $P_j$ are labeled with the same integer only when $i=j$ and they are the same piece. Let $G$ be an injection from finite sequences of integers to positive integers such that 
\begin{equation}\label{eg:Ggeqnone} G(n_1,n_2,n_3...n_k)> n_1 ~\text{for all finite sequences} ~n_1,n_2...n_k.
\end{equation} 

We now proceed to define the positive integer valued function $h$ discussed in the statement of \tref{main}. For any \om where $f(\omega)=0$, define $h(\omega)$ to be $g(1)$. Now fix $i>0$ and suppose $h$ is defined for every \om for which $f(\omega) < i$. We now suppose $f(\omega)=i$ and describe how we define $h(\omega)$. Select the least positive $m$ such that $f(T^m(\omega))\geq i$ (the existence of such an $m$ is guaranteed by (2) of \lref{fg}). Then $h(T(\omega)), h(T^2(\omega)), ... h(T^{m-1}( \omega))$ are already defined by induction. Now let $a_0, a_1, a_2, ... a_m$ be the integers labeling the pieces of $P_i$ containing $\omega, T(\omega), ... T^m(\omega)$ respectively. We define $h(\omega)$ to be 
\vskip .2 cm
$\\ h(\omega) := G(g(i+1), a_0, a_1, a_2, ... a_m, -1, h(T(\omega),h(T^2(\omega), ... h(T^{m-1} (\omega))$
\vskip .5cm
Note that since $-1$ is the only negative term in the definition of $h$, it serves as a comma separating the terms $a_j$ from the terms $h(\omega_i)$. In other words, since $G$ is an injection 

\begin{align} 
& \text{$h(\omega)$ tells you all the $a_i, 1 \leq i \leq m-1$, and all the} ~h(T^i(\omega)), 1\leq i\leq m-1,\notag \\
& \text{and the $-1$ lets you know which is which. (Recall that $i:=f(\omega)$ and $m$ is}\notag \\
& \text{the least positive integer such that $f(T^m(\omega)\geq{i}$)}\label{eg:htells}
\end{align} 

Actually in this case the -1 is pointless since you already know $m$ to be half the number of terms (including the -1) but we are including it so that we can generalize later.

Now the map 
$\\ \omega  \rightarrow ... h(T^{-2}(\omega),h(T^{-1}(\omega),h(\omega), h(T(\omega)), h(T^{2}(\omega), ... $
is a homomorphism from $T$ to a stationary process which is an isomorphism if it separates points (in which case the countable partition defined by ``$\omega_1$ is in the same piece of the partition as $\omega_2$ iff $h(\omega_1)=h(\omega_2)$" is called a generator of $T$). Thus the proof of \tref{main} will be complete when we show \eref{separate} and \eref{npastdet} below:

\begin{align}
&\text{The above stationary process separates points, i.e. for any two distinct } \notag\\
&\text{points \al and \be in \ombl, there is a $j$ such that}~ h(T^j(\alpha)) \neq h(T^j(\beta)).\label{eg:separate}\\
&\ \notag \\
&\text{If}~ h(T^{-1} (\omega))=r,~ \text{then}~ h(T^{-r}(\omega)), h(T^{1-r}(\omega)),... h(T^{-1}(\omega))~   
\text{determines}~ h(\omega). \label{eg:npastdet}
\end{align}

Proof of \eref{separate}: 
Let \omo and \omt be two distinct elements of \ombl. Then there exist a $j$ such that $P_j$ separates them and by (2) of \lref{fg}, there is a positive $k$ such that 
$f(T^{-k}(\omega_1))>j$. By selecting the smallest such $k$ and letting $f(T^{-k}(\omega_1)) =:ii>j$,  by \eref{htells}, $h(T^{-k}(\omega_1))$ encodes which piece of  $P_{ii}$ and hence which piece of $P_j$ (because the partitions are increasing) \omo is in and thus is different from $h(T^{-k}(\omega_2))$.

Proof of \eref{npastdet}: 
Let $\omega_1= T^{-1}(\omega)$ so that $h(\omega_1)=r$. Let $i:=f(\omega_1)$. By \eref{Ggeqnone}, $r=h(\omega_1)>g(i+1)$ so by (2) of \lref{fg} there is an element of  
$\{(T^{-r}(\omega), T^{1-r}(\omega),... T^{-1}(\omega))\}$
where $f$ takes on a value of more than $i+1$. Let $\omega_2 := T^{-k}(\omega)$ where $k$ is the minimum positive number with $f(T^{-k}(\omega))>i+1$ so if we let $f(T^{-k}(\omega))=:i_0$, then $i_0>i+1$. Then $-k\geq -r$. We will be done if we can show that 

\begin{equation}\label{eg:hom2dethom} h(\omega_2) ~\text{determines}~ h(\omega).
\end{equation} 

Now let $s$ be the least nonnegative number such that  $f(T^s)(\omega)\geq i_0$ which is greater than $i+1$. Since $i = f(T^{-1}(\omega))$, by (1) of \lref{fg}, $s > 0$ so \eref{hom2dethom} follows by \eref{htells} applied to $\omega_2$ where we replace $i$ with $i_0$. \hfill $\Box$

\begin{co}\label{co:more}
We actually proved more than we said we would. All we said we would prove is that $h(T^{-r}\omega), h(T^{1-r} \omega),... h(T^{-1} \omega)$ determines $h(\omega)$. We actually proved that one of those values alone encodes enough information to determine $h(\omega)$. \end{co}

\begin{co}\label{co:assumecg}
By proving the above theorem we have among other things reproved the countable generator theorem. We just defined $h$ as
$\\h(\omega):=G(g(i+1), a_0, a_1, a_2, ... a_m, -1, h(T(\omega), h(T^2(\omega), ... h(T^{m-1} (\omega))$
where $a_0, a_1, a_2, ... a_m$ are the integers labeling the pieces of $P_i$ containing 
$\\\omega,T(\omega), ... T^m(\omega)$ respectively.

 However, suppose we already have a finite or countable generator $P$ before starting this proof and labled the pieces of $P$ as integers. We can simply write 
\\$h(\omega):= G(g(i+1), a, h(T(\omega), h(T^2(\omega), ... h(T^{m-1} (\omega)) $
where $a$ is the element of $P$ containing \om if we also insist that 
$f(\omega)=0 \Rightarrow h(\omega)= G(g(1),a))$ 

\noindent because this defines $h$ to be finer than $P$ which we already know to generate, so $h$ generates the \sal and hence we don't need the $a_1,a_2,...a_{n-1}$.
\end{co}
\begin{co}\label{co:hientropy}
Now suppose $h$ has been defined as in \coref{assumecg}. The purpose of the following examples and analysis are to show that although our theorem says that the past determines the future in a very strong finitistic way, when you read the other way and look at how the future effects the past, the conditional measure on the present given the future has almost all the full entropy of the original process. The intuitive idea of this is that if you look at the future of the $h$ process, the only thing it encodes is the future of the $P$ process together with a knowledge of the future $f$ process and future $g$ process. The future $f$ and $g$ processes are determined by where you will be in the Rokhlin towers. Thus the only information you have about the past of the $P$ process given the future of the $h$ process, that you would not know just by looking at the future of the $P$ process, is where you will be in the Rokhlin towers, and if these towers are built with tiny and rapidly decreasing bases that is not very much information. {\bf Until we say otherwise, $h$ is defined as in \coref{assumecg}, not as in the proof of the theorem,} but of course the theorem still holds.
\end{co}

Example 1: Start off with the standard 2 shift with canonical 1/2, 1/2 generator whose pieces will be denoted by $H, T$, to suggest that we are looking at infinitely many independent flips of a fair coin ($H$ meaning heads and $T$ meaning tails.) Now cross that process with another aperiodic process with small entropy (you can even assume that the process we are crossing with has entropy zero.). We consider the proof of \tref{main} for this product process. Let $(a,b)$ be a generator of the small entropy process so that the entire process has generator $((a,H),(b,H),(a,T),(b,T))$ which we can write as (1,2,3,4). Now arrange that all our Rokhlin towers are measurable with respect to the small entropy factor so that $f$ and $g$ are measurable with respect to the small entropy factor. Then define $h$ as in \coref{assumecg} except make sure that $h$ takes on an odd number when \om is in 1 or 2 and an even number when \om is in 3 or 4 (which we can do by controlling our definition of $G$). Now when we look at the future $h$ process (i.e. $h(1),h(2). . .$), the only thing it encodes is the future $(H,T)$ process together with the values of $f$ and $g$ in the future. But the $f$ and $g$ processes are independent of the $(H,T)$ process, and since the future $H,T$ values are independent of the present $H$ or $T$ it follows that even conditioned on the future $h$ process, the present value of $h$ is odd with probability 1/2. This means that as we look backwards in time all the randomness of the $H,T$ process is still there.

Example 2: (generalization of previous example): Now start with an arbitrary $P$ process (we will assume a two set generator (1,2) for simplicity but the reader probably has enough imagination to figure out how we would handle an arbitrary generator.) Do the same as we did in the previous example, crossing with a aperiodic transformation of arbitrarily small entropy (perhaps 0) and arranging for $f$ and $g$ to be measurable with respect to that transformation and $h$ to be odd iff we are on the piece labeled ``1" of $P$. Then if we just look at the oddness and evenness of $h$ we see the original $P$ process and the probability of an odd number in the present given the future $h$ process is the same as the probability of a ``1" given the future $P$ process so that all the randomness of the $P$ process is maintained as we look backwards in time. 

Example 3:  (General case except we use two set generator for simplicity) Do the same as in the previous example again assuming generator (1,2) but this time don't bother crossing it with anything. By choosing your Rokhlin tower to have small and rapidly decreasing bases, we can arrange that a typical $f$ and $g$ name is exponentially big (i.e. that the $f$ and $g$ processes has small entropy.) Again let the evenness or oddness of the $h$ process read off the $P$ process. Given the future $h$ name, all we know is the future $f$ and $g$ names and future $P$ name. The theory of relative entropy (In particular Pinsker's formula) implies that the relative entropy of the $P$ process over the $f$ and $g$ processes is only slightly less than the entropy of the $P$ process. If, for example, the $P$ process has entropy $1/3$, then even when conditioned on the entire future $h$ process, the relative entropy of the $P$ process over the $f$ process is only slightly less than $1/3$, i.e. the expected entropy of the the two set partition 
\vskip .3cm
[the value of $h$ at time 0 is even, the value of $h$ at time 0 is odd] 
\vskip .3cm
\noindent given the future of $h$ is only slightly less than 1/3.
\begin{co}\label{co:opposite}
We now do the opposite extreme. We arrange for \tref{main} to hold in both directions. The above examples show that we can arrange for the past to determine the future in this very deterministic way while the past given the future can be made to have almost the full entropy of the process. We now point out that if we had wanted to, we could instead have had the theorem work in both directions, i.e. not only would an $n$ at time -1 have meant that $h$ at times $-n,...-2,-1$ have determined the $h$ at time 0, but also that an $n$ at time 1 would imply that $h$ at times $1,2,...n$ would determine $h$ at time 0. 
\end{co}
\noindent Accomplishing \coref{opposite}:
\vskip .3cm
 \noindent Simply replace the sentences
\vskip .3cm
\noindent``Select the least positive $m$ such that $f(T^m(\omega))\geq i+1 \hskip.03cm    $ (the existence of such an $m$ is guaranteed by (2) of \lref{fg}). Then $h(T(\omega)), h(T^2(\omega)), ... h(T^{m-1}( \omega))$ are already defined by induction."
\vskip .3cm
\noindent in the proof of the \tref{main} with the sentences
\vskip .3cm

\noindent``Select the least positive $m$ such that $f(T^m(\omega))\geq i+1$ and the least positive $n$ such that 
$f(T^{-n}(\omega))\geq i+1$(the existence of such an $m$ and $n$ are guaranteed by \lref{fg}). Then $h(T^{1-n}( \omega)). . .h(T^{-1}( \omega)),h(T(\omega)), h(T^2(\omega)), � h(T^{m-1}( \omega))$ are already defined by induction."
\vskip .3cm
\noindent and then define $h(\omega)$ to be
\vskip .3cm
\noindent $G((g(i+1)+1, a_0, a_1, a_2,. . . a_m, -1 , 
\\ h(T^{1-n}( \omega)). . .h(T^{-1}( \omega)),h(T(\omega)), h(T^2(\omega)), . . . h(T^{m-1}( \omega)))$
\vskip .3cm
\noindent and then you can carry out the proof in both directions. There is no reason to add more $a_i$s because their only purpose is to assure that $h$ is a generator.

\section{Uncountable Partitions}

\begin{co}\label{co:myuncount}
We now consider processes on an uncountable alphabet. The following is a known trick (this is \cite{Rokhlin} already mentioned in the introduction). Suppose we have a process with a generator $P$ (finitely or countably infinite) and we want to get an uncountable generator which behaves perversely. Just let your partition be $P^N$ and the piece of the partition that \om is in is $(p_0,p_1,p_2,...)$ where $p_i$ is the piece of $P$ containing $T^i(\omega)$. Then the term at time -1 completely determines the term at time 0 but when reading from future to past you get a Markov chain with all the randomness of the original process. However we have found a particularly perverse example of this phenomenon. 
\end{co}

To motivate this example consider a 2 dimensional process in which each lattice point of the plane is endowed with a random variable which takes on 1 with probability 2/3 and 0 with probability 1/3 and suppose these variables are all independent. This completely defines a process. Here is another way to define the same process. Just say that 
\begin{align}
&\text{running the process backwards, the columns form a stationary process} \notag\\
&\text{which is in fact a Markov chain in which the conditional probability} \notag\\
&\text{of the 0 column given the 1 column is always the 1/3, 2/3 i.i.d. measure.} \label{eg:iid}
\end{align}

Hence it is impossible to obtain a process which is defined by \eref{iid} and still have the -1 column determine the 0 column. The following example shows that you can almost do that.

Define a (generally non-stationary) sequence of random variables 

$. . .,X_{-2}, X_{-1}, X_0, X_1, X_2,. . .$

\noindent to be blue if they are independent of each other and each one takes on one of the following two distributions; 0 with probability 1/3 and 1 with probability 2/3, or 1 with probability 1/3 and 0 with probability 2/3. Since a blue process is not in general stationary, the entropy of a blue process is not defined but it is intuitively obvious that a blue process from an entropy standpoint is the same as the 1/3, 2/3 ~i.i.d. process. Hence, given the previous paragraph it is rather surprising that we can define a process on the 2 dimensional lattice points such that the columns form a stationary process, the -1 column determines the 0 column with probability 1, and yet the column process is a backwards Markov chain where the conditional probability of the 0 column given the 1 column is always blue. Although it is surprising that we can do that it is also easy so we suggest that the reader try to do it himself before reading onward.

We define the process by defining the backwards Markov Chain. Let $f$ be any infinite to 1 map from the integers onto the integers. We assume we know the 1 column and define the measure on the 0 column (independent of the 2,3,... columns). We simply define it to be 
\begin{align}
&\text{the blue process which on point $(0,i)$ takes on the} \notag\\
&\text{same value as $(1,f(i))$ with probability 2/3} \notag\\
&\text{and the opposite value with probability 1/3.}\label{eg:markov} 
\end{align}
The -1 column determines the 0 column because if you know the -1 column and you want to determine the value at $(0,i)$ just look at the infinite set of values on terms $(-1,j)$ where $f(j) = i$ and by the strong law of large numbers 2/3 of them will be 1 iff the term at $(0,i)$ is 1. To rigorously define this process we have to define a stationary measure on the columns but the theory for doing that is analogous to the theory for finite state Markov processes and we leave that to appendix. Also analogous is the proof that this is a mixing Markov process and that such processes are Bernoulli (all proved in the appendix).

\section{Uniform Martingales}
This section modifies the proof of \tref{main} to prove \tref{ext} using  \dref{umt}. \dref{um} is trivially satisfied for the process guaranteed by \tref{main} because for any integer $a$ there is a fixed time past which you will know precisely whether or not the integer at time 0 is $a$ given that much past (actually the theorem does not say that but if you look at the proof you will see that not only is it the case that if there is a 5 at time -1 then the process at times -5,-4,-3,-2,-1 determines the value at time 0, but furthermore if there is a 5 at time 0 then the process at time -5,-4,-3,-2,-1 determines that there is a 5 at time 0.) However \dref{umt} fails for this process because given the entire past, the measure at time 0 puts all its mass at one number and if that number is huge it will take a long time before you can even suspect what it is. It is \dref{umt} which is of interest here.

It was shown that a specific finite state process called the $T, T^{-1}$ process could be extended to a finite state uniform martingale in \cite{kal}, and then it was shown that every zero entropy process could be extended to a finite state uniform martingale in \cite{wei}. In both cases essentially the same technique was used. We will use that technique again in this paper to show that by modifying  \tref{main} we can establish that every aperiodic stationary process can be extended to a uniform martingale on a countable state space.

The basic method used in those two papers and in this one for constructing a uniform martingale is as follows. First we construct two jointly distributed processes, one an i.i.d. process $. . .,a_{-2}, a_{-1}, a_0, a_1, a_2,. . .$ of positive integers called the lookback process. The other an arbitrary process,
 
\noindent $. . .,b_{-2}, b_{-1}, b_0, b_1, b_2,. . .$ on a finite or countably infinite alphabet and we arrange that these processes are jointly distributed so that

\begin{align} 
&\text{ The $a_i$ process is an i.i.d. process}\label{eg:lbiid} \\
&\text{	Each $a_i$  is independent of all the following random variables jointly:}\label{eg:lbind}\\
&\text{ all $a_j, j<i$ and all $b_j, j<i.$}\notag\\
&\text{	Each $b_i$ is determined by $a_i, b_{i-1}, b_{i-2}, . . .  b_{i-a_i}.$}\label{eg:lbdet}
\end{align}

\begin{co}\label{co:silly} It is obvious that \eref{lbind} immediately implies \eref{lbiid} but we will often say, ``Suppose we have \eref{lbiid}, \eref{lbind} and \eref{lbdet}" because it is useful to get the reader to keep \eref{lbiid} in mind.
\end{co}

\begin{lemma}\label{lem:lbsgiveum} If we can establish $. . .,a_{-2}, a_{-1}, a_0, a_1, a_2,. . .$ 
\\ and $  . . .,b_{-2}, b_{-1}, b_0, b_1, b_2,. . .$ processes obeying \eref{lbiid},\eref{lbind},  
\\ and \eref{lbdet}, then $ . . .,b_{-2}, b_{-1}, b_0, b_1, b_2,. . . $ is a uniform martingale.\end{lemma}

\begin{proof}

Let \mut be the measure on $b_0$ given the $n$ past (the $n$ past  $:=b_{-1}, b_{-2}... b_{-n}$) and let 
\nut be the measure on $b_0$ given the entire past (the entire past $:=b_{-1}, b_{-2}...$) . To establish the lemma we must couple \mut and \nut so that they usually agree where exactly how $``$usually" depends only on $n$ (Coupling two measures means establishing a measure on the product space of the two spaces with the two measures as marginals. Joining is a special case of coupling). We are now going to couple \mut and \nutl. 

\begin{align} 
&\text{1) Fix $b_{-1},b_{-2}. . .b_{- n}$ in accordance with the measure on the $n$ past}\notag \\ 
& \text{2) Select $b_{-n-1},b_{-n-2},. . . $ in accordance with the measure on the} \notag \\
& \text{entire past from $-n-1$ onwards given that the $n$ past is }\notag\\ 
& b_{-1},b_{-2}. . .b_{- n}.\notag\\
& \text{3)	Use the same measure as was used in (2) to select $B_{-n-1},B_{-n-2}...$}\notag\\
& \text{and chose them to be independent of $b_{-n-1},b_{-n-2},. . .$ conditioned on} \notag\\
& b_{-1},b_{-2}. . .b_{-n}\notag\\
& \text{4)	Now we have two complete $b$ pasts}\notag\\
& \notag\\
&b_{-1},b_{-2}. . .b_{- n}, b_{-n-1},b_{-n-2},. . .\notag\\
&b_{-1},b_{-2}. . .b_{- n}, B_{-n-1},B_{-n-1}. . .\notag\\
&\notag\\
& \text{Select $a_{-1},a_{-2}. . .$  so that its joint distribution with} \notag\\
&  b_{-1},b_{-2}. . .b_{- n}, b_{-n-1},b_{-n-2},. . . \notag\\
& \text{obeys \eref{lbiid},\eref{lbind} and \eref{lbdet} and then select}\notag\\
& A_{-1}, A_{-2},. . . ~\text{so that its joint distribution with} \notag\\
& b_{-1},b_{-2}. . .b_{- n}, B_{-n-1},B_{-n-1}. . .\notag\\
& \text{is the same as the joint distribution of}~ a_{-1},a_{-2}. . . \notag\\
& \text{and}~ b_{-1},b_{-2}. . .b_{- n}, b_{-n-1},b_{-n-2},. . .\notag\\
& \text{5) Now we select $A_0$ and $a_0$. First choose $a_0$ independent of everything}\notag\\
& \text{chosen so far. Note that all we need to know about $A_{0}$ is that it is}\notag\\
& \text{independent of all other $A_{- i}$, of$~b_{-1},b_{-2}...,b_{-n}$, and of all $~B_{-i}$}  \notag\\
& \text{so we can choose $A_{0}$ to equal $a_0$.} \notag
\end{align}

Now focus on the second process (the one which uses $``B"$s) and consider the measure \mut on $b_0$ conditioned on only $b_{-1}, b_{-2}, ...b_{-n}$ (in other words integrate out over all possible continuations $B_{-n-1},B_{-n-2}...$) and then on the first process (the all $``b"$ process) and consider the measure \nut on $b_0$ conditioned on the entire past. If $a_0<n$ the two values of $b_0$ are identical (by (5) and \eref{lbdet}). We have established a coupling of \mut and \nut which agree unless $a_{0}> n$  and thus the lemma is proved.
\end{proof}

\begin{corollary}\label{cor:umpair}[of the proof of \lref{lbsgiveum}] If we can establish 
$\\. . .,a_{-2}, a_{-1}, a_0, a_1, a_2,. . .$ and  $. . .,b_{-2}, b_{-1}, b_0, b_1, b_2,. . .$
 processes obeying \eref{lbiid},\eref{lbind} and \eref{lbdet}, 
 then $ \\. . .(a_{-2},b_{-2}), (a_{-1},b_{-1}), (a_0,b_0), (a_1,b_1), (a_2,b_2) . . .$ is a uniform martingale.
\end{corollary}

\begin{co}\label{co:drop}
The reason this corollary was never mentioned in \cite{kal} or \cite{wei} was that we did not want to keep the $a_i$ because that would cause our alphabet to be countably infinite and we wanted it to be finite but in this paper it will turn out that the $. . .,b_{-2}, b_{-1}, b_0, b_1, b_2,. . .$
process already has an infinite alphabet so it does not make any difference whether we keep the  $. . .,a_{-2}, a_{-1}, a_0, a_1, a_2,. . .$ or drop them. Our only reason for pointing out that you can drop them in this paper is that it may help the reader in future research for solving some of the currently open problems that we mention in this paper. 
\end{co}
{\it Preparation for the proof of \tref{ext} a:}

Here is how we will proceed. This is exactly the same procedure that was used in \cite{kal} and \cite{wei}. We will start with our arbitrary non-periodic transformation.

Step 1: We need to alter \tref{main}. We will then endow our non-periodic process with a countable generator using the altered form of \tref{main}. This gives us a stationary process $ . . .,c_{-2}, c_{-1}, c_0, c_1, c_2,. . .$
on a countable alphabet which is isomorphic to the given transformation. This step uses Alpern towers with nested bases.

\vskip .5cm

Step 2: We will then cross that process with a lookback process  
\\$. . .,a_{-2}, a_{-1}, a_0, a_1, a_2,. . .$ 

\vskip .5cm

Step 3: We will then change some of the values $c_i$ to a new letter called $``$question mark" notated by the symbol $``?"$. This is done in infinitely many stages. After making these changes $. . .,c_{-2}, c_{-1}, c_0, c_1, c_2,. . .$ will now become another process $. . .,b_{-2}, b_{-1}, b_0, b_1, b_2,. . .$
, i.e. each $b_i$ is either $c_i$ or ``?". 

\vskip .5cm

Step 4: We will establish \eref{lbiid},\eref{lbind}, and \eref{lbdet} for $. . .,a_{-2}, a_{-1}, a_0, a_1, a_2,. . .$ and $ . . .,b_{-2}, b_{-1}, b_0, b_1, b_2,. . .$
 thereby establishing that $ . . .,b_{-2}, b_{-1}, b_0, b_1, b_2,. . .$is a uniform martingale by \lref{lbsgiveum}. 

\vskip .5cm

Step 5: We will show that  $. . .,b_{-2}, b_{-1}, b_0, b_1, b_2,. . .$
is an extension of 
\\$. . .,c_{-2}, c_{-1}, c_0, c_1, c_2,. . .$, i.e. that from a given $b$ name we can recover the $c$ name, i.e. that by looking at the entire $b$ name we can determine what the question marks were before they were converted to question marks. We will then be able to conclude that $. . .,c_{-2}, c_{-1}, c_0, c_1, c_2,. . .$
can be extended to a uniform martingale on a countable alphabet.

\vskip .5cm

In this paper, as is the case in every case where this procedure has been used, in order to establish step 5 it is necessary that the past not only determine the present but that furthermore a random sampling of the past is sufficient to determine the present. This means that it is necessary when establishing step 1 for us to get the past to determine the present with substantial redundancy (i.e. we need to arrange that many pieces of past determine the present). \tref{main} does not provide sufficient redundancy so we will have to modify the proof of it to get that redundancy. 

\vskip .5cm
Proof of \tref{ext} a:
\vskip .5cm

{\it Carrying out step 1}: For those who read \cite{kal} or \cite{wei} this is really the only step you should have to read. Steps 2 through 5 are just repeating the procedure in those papers. 

The whole goal of \tref{main} is to arrange that if $h(T^{-1} (\omega))=k$ then if you look back $k$ steps you will be able to determine $h(\omega)$. Here we are no longer interested in looking at $T^{-1}(\omega)$ but rather on \om itself so (1) of \lref{fg} is no longer of interest to us. On the other hand (2) of \lref{fg} says that 
for all \om and all nonnegative integers $i$, there is a member of 
$\{T(\omega),T^2(\omega) . . . T^{g(i)}(\omega)\}$ where $f$ takes on a value greater than $i$
but here we would like to replace that with

\begin{align}
& \text{ for all \om and all nonnegative integers $i$, there is}\notag \\
& \text{a member of}~  \{T(\omega),T^2(\omega) . . . T^{g(i)}(\omega)\} \notag\\
&\text{where $f$ takes on} ~i+1\label{eg:newfg}.
\end{align} 

Since we are dropping (1) of \lref{fg} we can simplify the definition of $f$. Just define it to be the largest $i$ such that $\omega \in B_i$ where $B_i$ is the base of the $``i"$th tower letting it be 0 if there is no such $i$ (Again there is a largest one by Borel Cantelli). Now instead of defining $g(\omega)$ to be $n_{i+1}+1$, define it to be $2n_{i+1}+2$. 

$n_{i+1}+1$ would be sufficient if all we wanted was a $k$ where $f(T^k(\omega))> i$ because in any $n_{i+1}+1$ consecutive terms there must be a term where you are in $B_{i+1}$ of \lref{rohlin} but unfortunately you might also be in $B_j$ for $j>i+1$ which would imply that you are in $B_{i+2}$ because in this proof we assume the $B_j$ are nested. But if that happens then in the following $n_{i+1}+1$ terms you cannot conceivably be in $B_{i+2}$ because the height of the $i+2$ Alpern tower is too big. Hence somewhere in those $2n_{i+1}+2$ terms you must be in $B_{i+1} \setminus B_{i+2}$ and for that term $f$ takes on exactly $i+1$.  \eref{newfg} is established. {\it This is the only use of nested bases in this paper}.

In the proof of \tref{main} we set $h(\omega)$ equal to 
\vskip .5cm
$G((g(i+1)+1,\textsl{} a_0, a_1, a_2,. . . a_m, -1 , h(T(\omega), h(T^2(\omega)), ... h(T^{m-1}( \omega)) ))$ 
\vskip .5cm
\noindent where $m$ was chosen to be the least number such that $f(T^m(\omega))\geq i$ and $i$ is chosen so that $f(\omega)=i$. We now alter that definition. In this section we can just let $h(\omega)$  be any fixed positive integer if $f(\omega) = 0$ (if you want us to be precise let $h(\omega)=1$ if $f(\omega)=0)$. Now we let $i:=f(\omega)$, assume $i>0$ and assume $h(\omega')$ has been defined for all $\omega'$ with $f(\omega')<i$. We start with a rapidly increasing function, say $i!!$. 

\begin{defn} \label{defn:Ai}$A_i = i!!g(i+1).$
\end{defn}

\begin{co}\label{co:doublefact} By \eref{newfg} for any $\omega$ and any $A_i$ consecutive integers, $k,k+1...k+A_i-1$ there are at least $i!!$ integers $j$ such that $f(T^j(\omega))=i+2$.
\end{co} 

We now redefine $h$.

\begin{defn}\label{defn:newh} Recall that $f(\omega)=i$. Let $\omega_1, \omega_2,... \omega_r$ be the subsequence of all elements $\omega'$ of 
$\{T(\omega), h(T^2(\omega)), ... h(T^{A_i}( \omega)\}$\text~ such that $f(\omega')< i$ so that $h$ has already been defined. For $1 \leq i \leq  r$ let $Q(i,\omega)$ be the value of $j$ such that $T^j(\omega) =\omega_i$. Let $P_k$ be an increasing sequence of partitions which separates points where the pieces are integers as before. Let $a_j$ be the element of $P_i$ containing $T^j(\omega)$ and define 
\vskip .2cm
$h(\omega) := G(a_0, a_1, a_2,. . . a_{A_i}, -1 , Q(1,\omega),h(\omega_1), Q(2,\omega),h(\omega_2),...Q(r,\omega),h(\omega_r))$
\end{defn} 
$h$ generates for the same reason as before. Recall that in our original definition of $h$, if $f(T^{-1}(\omega))=k$ then in the last $k$ terms there is a term which tells you every $h(T^i(\omega))$ at least until $i$ reaches a $j$ where $f(T^j(\omega))> k+1$  and in particular tells you the value of $h(\omega)$. Now we do not focus on $T^{-1}(\omega)$ but rather on $\omega$ itself. Furthermore we are no longer interested in what you learn when you look back $h(\omega)$ but rather we are looking at what you learn when you look forward $A_i$ where $i:=f(\omega)$. Immediately from our new definition of $h$,

\begin{align}
&\text{if $f(\omega)=i$ then for any $\omega'$ among the terms} \notag\\ 
& T(\omega),T^2(\omega)...T^{A_i}(\omega) \text{ for which } f(\omega)<i, \notag\\
& h(\omega) \text{ determines } h(\omega').\label{eg:forwardid}
\end{align}

\coref{doublefact} and \eref{forwardid} provides the redundancy needed to carry out our construction.

\begin{notation}\label{not:identity} When we say \om knows it's identity we mean that $h(\omega)$ is determined. For $a<b,\hskip .08 cm T^a(\omega)$ tells $T^b(\omega)$ its identity means that $h(T^a(\omega))$ encodes $h(T^b(\omega))$ or we may just say that $a$ tells $b$ its identity when $T$ and \om are understood.
\end{notation}

\vskip.5cm

{\it Carrying out step 2}: Let $c_i = h(T^i\omega)$ which makes $c_i$ a random variable. Since as in \tref{main} $h$ is a generator for the transformation, the shift on the $c_i$ process is isomorphic to $T$. We now create an extension of $T$ by attaching $. . .,c_{-2}, c_{-1}, c_0, c_1, c_2,. . .$ to a lookback process  $. . .,a_{-2}, a_{-1}, a_0, a_1, a_2,. . .$ which we choose to be an i.i.d. process which is completely independent of $. . .,c_{-2}, c_{-1}, c_0, c_1, c_2,. . .$, All that remains is to describe the distribution of $a_0$.  

\begin{align}
& \text{ Distribution of $a_0$:  We let $a_0$ take on the value $A(i)$} \notag \\
& \text{ with probability $1/2^i$ and then the distribution of the}\notag \\
&. . .,a_{-2}, a_{-1}, a_0, a_1, a_2,. . .  \text{ process is completely defined by the}\notag\\
&\text{fact that it is i.i.d.} \label{eg:2toi}
\end{align} 

{\it Carrying out step 3:} 
Take the independent product of the $\\. . .,a_{-2}, a_{-1}, a_0, a_1, a_2,. . . . $ process and the  $. .,c_{-2}, c_{-1}, c_0, c_1, c_2,. . .$
process and write it as $~~...(a_{-2},c_{-2}),(a_{-1},c_{-1}),(a_0,c_0),(a_1,c_1),(a_2,c_2),...
~\\$ and call that the 0 process. 

We now define a sequence of processes starting with the 0 process by successively turning more and more of the $c_i$ to question marks while never altering the $a_i$ process. 
To obtain the 1 process, let $c1_i$ = ``?" if none of $c_{i-1},c_{i-2}....c_{i-a_i}$. 
determines the identity of $c_i$. Otherwise, let $c1_i=c_i$. 

After changing these unidentified $c_i$ to question marks, $. . .,c_{-2}, c_{-1}, c_0, c_1, c_2,. . .$
 turns to $. . .,c1_{-2}, c1_{-1}, c1_0, c1_1, c1_2,. . .$,
 i.e. each $c1_i$ is either $c_i$ or ``?". The 1 process is $~\\...(a_{-2},c1_{-2}),(a_{-1},c1_{-1}),(a_0,c1_0),(a_1,c1_1),(a_2,c1_2),...
~\\$

To obtain the 2 process take any $i$ where $c1_i$ has not yet turned to a ``?" and turn it into ``?" none of $c1_{i-1},c1_{i-2}...,c1_{i-a_i}$ determines its identity to get $. . .,c2_{-2}, c2_{-1}, c2_0, c2_1, c2_2,. . .$. so that the 2 process is $~\\...(a_{-2},c2_{-2}),(a_{-1},c2_{-1}),(a_0,c2_0),(a_1,c2_1),(a_2,c2_2),...
~\\$  

Clearly if $c1_i$ is ``?" so is $c2_i$ because you have less information at the second stage than you do at the first (a ``?" does not determine the identity of anything) but new question marks may occur at the second stage that were not question marks at the first stage because a term $c_k$ which allowed a given $c_i$ to determine its identity may have changed to a question mark at the first stage and hence was not available at the second stage to help $c_i$ determine its identity.

Continue this procedure to obtain the 3 process, 4 process etc. turning more and more terms into question marks. At every stage we know all of the $a_i$. Define $. . .,b_{-2}, b_{-1}, b_0, b_1, b_2,. . .$ by $b_i =~$``?" if $cn_i = ~$ ``?" for any $n$ and $b_i = c_i$ otherwise.  This limiting process, together with the $a_i$ process is to be called the final process. It is not immediately obvious that it is not the case that all $b_i$ are ``?" but we will ultimately see that many of the $b_i$ are $c_i$. This is exactly the procedure used in \cite{kal} and \cite{wei} and just as in those cases the important issue was getting the past to redundantly tell you information about the present. The final process is the product of 
$~\\. . .,a_{-2}, a_{-1}, a_0, a_1, a_2,. . .$ and 
$~\\. . .,b_{-2}, b_{-1}, b_0, b_1, b_2,. . .$ which can also be written as
$~\\...(a_{-2},b_{-2}),(a_{-1},b_{-1}),(a_0,b_0),(a_1,b_1),(a_2,b_2),...
~\\$ 
{\it Carrying out step 4}: We now wish to establish \eref{lbiid}, \eref{lbind} and \eref{lbdet} for the final process. \eref{lbiid} is given. To see \eref{lbind}, i.e. that $a_0$ is not only independent of all previous $a_i$ but also of all previous $b_i$, just note that $a_0$ is independent of all the previous $a_i$ and $c_i$ in the 0 process and that all previous $b_i$ are determined by the previous $a_i$ and the previous $c_i$. We now need only establish \eref{lbdet} which we do by showing that a given $b_k$ is a question mark iff it is impossible to determine its identity by looking back $a_k$ in the $\{b_i\}$ process (that means looking at the $k-1,k-2,...k-a_k$ terms) and that otherwise it is still $c_k$. Your final $b_k$ will be a ``?" iff it became such at a finite stage iff in one of the intermediate processes everyone amongst those $a_k$ terms who determines its identity died (i.e.was a question mark). If not then at the end there is still someone left to tell him his identity.

{\it Carrying out step 5}: All that remains is to show that the $b$ process determines the $c$ process (i.e. that the $b$ process extends the $c$ process, i.e. that by looking at the entire $b$ process you can determine what the question marks were before they turned to question marks.)

We select a fixed $\omega$, say $\omega_0$
and then look at the $a_i, b_i$ and $c_i$ processes for $\omega_0$. It suffices to show that the $b$ process determine $c_0$ because if you can prove that you can determine $c_0$ than you can determine all $c_j$ by just translating the argument by $j$. 

Select $\epsilon > 0$ and after selecting \ep select a fixed huge number $i$ and in particular make sure that $i>f(\omega_0)$. Break the negative numbers $-j$ into the following classes.

\begin{align}
&\text{Class 1:} ~j \leq A_i,\notag\\ 
&\text{Class k, k$>1$:}  ~A_{i+k-2} < j \leq A_{i+k-1}.\notag
\end{align}

\vskip 0.5cm

\noindent By \coref{doublefact} and the fact that $A_{k+1}>2A_k$ for all $k$, it follows that for all k$>$2 there are at least  $(i+k-2)!!$ integers $-j$ in class k such that  $f(T^{-j}(\omega_0)) = i+k$. Since $A_i>A_{i-1}$, in class 1 there are at least $(i-1)!!$ integers $-j$ such that $f(T^{-j}(\omega_0))=i+1$. For all k including k=1, refer to all the $-j$ in class k such that $f(T^{-j}(\omega_0)) = i+k$ as the special class k terms.

Define event k for all $k \geq1$ to be that there is a special class k term $j$ such that $a_{-j} = A_{i+k}$ (recall that for any particular $j$ this has probability $1/2^{i+k}$). Then since, event 1, event 2, . . . have probabilities rapidly approaching 1 as $k$ approaches infinity (because there are so many special class k terms), the probability they all happen rapidly approaches 1 as $i$ approaches infinity. Since $i$ was chosen after \ep we can assume
\begin{equation}\label{eg:allhap} \text{they all happen with probability}~>1-\epsilon 
\end{equation}
 Suppose they all happen.  Let $j_k$ be a specific class $k$ term for which 

\begin{align}
&a_{-j} = A_{i+k}\notag\\ 
&\text{and because it is special}\notag\\
&f(T^{-j_{\hskip.01cm k}}(\omega_0)) = i+k.\notag\\
\end{align}

Since 
\vskip .5 cm
\noindent $j_{k+1}$ is in class k+1, $0<j_k<j_{k+1} \leq A_{i+k-1}<A_{i+k+1}$ and $f(T^{-j_{k+1}}(\omega_0))=i+k+1$ for every $k$, Applying \eref{forwardid}, replacing $T^{-j_{k+1}}(\omega_0)$ for \om, $i+k+1$ for $i$, and $A_{i+k+1}$ for $A_i$, 
\vskip .5 cm
\noindent we get that 
\vskip .5 cm 
\noindent $h(T^{-j_{\hskip 0.01cm k+1}}(\omega_0))$ determines $h(T^{-j_{\hskip .01cm k}}(\omega_0))$.
Now fix k and suppose we are going from the 0 to the 1 process and we want to know if $C1_{-j_k}$ is a question mark. $f(T^{-j_k}\omega_0) = A_{i+k}$ and since $j_{k+1}<A_{i+k-1}$ when $-j_k$ looks back $A_{i+k}$ it  sees $-j_{k+1}$ so $-j_k$ does not turn to a question mark for any $k$. Using this reasoning over and over, $c_{-j_{\raise-2pt\hbox{$\scriptscriptstyle k$}}}$ never turns to a question mark so $b_{-j_{\hskip 0.01cm k}}=c_{-j_{\hskip 0.01cm k}}$ for all $k$. By \eref{forwardid} $-j_1$ gives the identity of 0. Therefore by \eref{allhap} $c_0$ is determined with probability more than 1-\ep and since \ep is arbitraty $c_0$ is determined with probability 1. \hfill $\Box$ 

Proof of \tref{ext} b:
We just finished carrying out step 5 where we showed that from $. . .,c_{-2}, c_{-1}, c_0, c_1, c_2,. . .$ we can recover $ . . .,b_{-2}, b_{-1}, b_0, b_1, b_2,. . .$ and since $. . .,a_{-2}, a_{-1}, a_0, a_1, a_2,. . .$ is the same for both the 0 process and the final process, we can recover the 0  process from the final process. The final process was chosen as a function (in fact a factor) of the 0 process. Hence they are isomorphic. By \cref{umpair} the final process is a uniform martingale. Thus we have proved that if you cross a nonperiondic process with a i.i.d. process which has the probability law,
$(1/2,1/4,1/8...)$ you get a process which is isomorphic to a uniform martingale on a countable state space. The Ornstein isomorphism theorem says that every two independent processes with the same entropy are isomorphic even if that entropy is infinity (entropy 0 for a Bernoulli process is impossible except in a trivial case). Any such nonzero entropy (including infinity) can be achieved as the entropy of a countable partition whose probabilities are a decreasing sequence of positive reals. Hence every nontrivial Bernoulli process is isomorphic to a i.i.d. process on a countable alphabet with probability law $(a,b,c,...)$ where $a,b,c, ...$ is a decreasing sequence of positive real numbers and the only way we used the precise sequence $(1/2,1/4,1/8...)$ is that $i!!$ is much larger than the reciprocal of the $i^\text{th}$ term of that sequence so to handle the general case just carry out the exact same proof except use $(a,b,c,...)$  as your lookback probabilities and  instead of $i!!$ use a function that grows much faster than the reciprocals of that sequence. \hfill $\Box$ 
\section{Appendix} 

Our goal is to establish that the Markov chain of \eref{markov} has a unique stationary measure and that the resulting stationary process is Bernoulli. Keep in mind throughout that the Markov chain runs backwards in time.

\begin{defn} \label{defn:bluetr}
Throughout this section we assume the transition probabilities of the backwards Markov chain of  \eref{markov}. We will call these the blue transition probabilities. 
\end{defn} 

\begin{defn} \label{defn:xij}
Let $\{X(i,j): i, j~ \text{integers}\}$ be the collection of 0,1 valued random variables for which we are trying to find their joint distribution. \end{defn}

By the Caratheodory extension theorem it suffices to establish the joint distribution of {$X(i,j):-n \leq i \leq n,-n \leq j \leq n$} for any $n$ (assuming these joint distributions for different $n$ are consistent with each other.) The way you get existence and uniqueness of a stationary measure for a finite state Markov chain is to prove the renewal theorem and from that it is trivial. Here we do essentially the same thing. We prove an analog for the renewal theorem with essentially the same proof as the proof of the renewal theorem and then existence and uniqueness of a stationary measure follows by essentially the same proof. 

\begin{theorem}\label{thm:ren}
Analog of the Renewal theorem:

Let $\{a(i,j):-n \leq i \leq n,-n \leq j \leq n\}$ be a specific square array of 0s and 1s (i.e. $a(i,j) \in \{0,1\}~ for -n \leq i \leq n,-n \leq j \leq n$). Then for every $\epsilon >0$ there exists an $m$ such that for integers $M>m$ and $N>m$ and any two doubly infinite sequences of 0s and 1s
$A(i,M),-\infty < i <  \infty,$ and
$B(i,N),-\infty < i < \infty, $ we have that 
$| P(X(i,j)=a(i,j),-n \leq i \leq n,-n \leq j \leq n| X(i,M) = A(i,M),-\infty < i < \infty) - P(X(i,j)=a(i,j),-n \leq i \leq n,-n \leq j \leq n| X(i,N) = B(i,N),-\infty < i <  \infty) | <\epsilon $
\end{theorem}

For exactly the same reason as in the finite case, if we can prove \tref{ren} then there is a unique stationary measure for the Markov chain. 

\begin{proof}

We use one of the proofs that is used in the finite case. Let $X(i,j)$ be the Markov chain with the blue transition probabilities given running from the $M$ column back given 

$X(i,M) = A(i,M):-\infty < i <  \infty$ 

\noindent and let $Y(i,j)$ be the Markov chain with the blue transition probabilities from the $N$ column back given

$Y(i,N) = B(i,N):-\infty < i < \infty$

The proof of \tref{ren} will be complete if so long as $m$ is sufficiently large we can couple the processes $X(i,j)$ and $Y(i,j)$ so that with probability $>$ 1- \epl, the two processes are identical for all $1\leq i \leq n, 1\leq j \leq n$. 

\begin{defn} \label{defn:f}
$f$  is the infinite to one map from the integers to the integers which we use to define the Markov chain.
\end{defn}  

\begin{defn} \label{defn:father}
For every element of $(a,b)  \in \mathbb{Z}^2 $ we say that $(a+1,f(b))$ is the father of $(a,b)$
\end{defn}  
\begin{defn} \label{defn: ancestor}
Ancestor is the transitive closure of the relationship ``father".
(e.g.if $a$ is the father of $b$ is the father of $c$ is the father of $d$ then $a$ is an ancestor of $d$)
\end{defn}

\begin{lemma}\label{lem:unique} 
If $a<b$ are integers and $c$ is an integer then there is $a$ unique $d$ such that $(b,d)$ is an ancestor of $(a,c)$.
\end{lemma}
\begin{proof} This follows trivially from the fact that every $(e,f)$ has exactly one father and that father has $x$ coordinate $e+1$.
\end{proof}

Conclusion of proof of \tref{ren}: 
We now describe the coupling. Assume WLOG that $M \leq N$. Recall that the Markov chain runs backwards (In the forwards direction each column determines the next.). Start with the condition 
$Y(i,N) = B(i,N):-\infty  < i < \infty$ and run that process until the second coordinate reaches $M$. Then start running the process starting $X(i,M) = A(i,M):-\infty<i< \infty$. We start letting $X(i,M), -\infty <i<\infty$ be independent of $Y(i,M) -\infty < i < \infty$. For a given $(j,M)$ the two processes may or may not exhibit the same number (either both 1 or both 0). If they do then they have the same probability law for $(k,M-1)$ for any $k$ where $f(k)=j$ so we can couple them to be the same for $(k,M-1)$ for any such $k$. Otherwise for such a $k$ we couple them independently and that gives the probability that they end up the same at $(k,M-1)$ to be 4/9. We couple each coordinate of the $M-1$ column independently. Continue coupling in this manner.  In other words if a given father is the same for both processes then his son is the same for both processes and if the fathers differ then the sons are the same with probability 4/9. We continue that coupling from then on making all couplings of a given column independent given the next column (keep in mind that the next column is actually the previous column from the standpoint of the backwards running Markov chain). For a given $(i,j):-n\leq i\leq n,-n\leq j\leq n$ there is a unique $k$ such that $(k,M)$ is an ancestor of $(i,j)$ so if the two values on $(k,M)$ are the same they will remain the same for the son, and hence for the grandson etc. and hence ultimately they will be the same for $(i, j)$. Even if they differ on $(k,M)$ they have a 4/9 chance that they will be the same on the son of $(k,M)$ and then if they differ on the son they have a 4/9 chance of agreeing on the grandson etc. Once they agree they will continue to agree so since there are many generations between $(M,k)$ and $(i,j)$ they will almost certainly agree on $(i,j)$.  In fact even if they don't agree by time the second coordinate reaches $m$ if $m$ is large there is a lot of time for them to agree after that. Since this is true  for all $(i,j) -n\leq i\leq n,-n\leq j\leq n$ it follows that they agree on that square with a probability bounded below by a function of $m$ which goes to 1 as $m$ approaches $\infty$.
\end{proof}

\begin{theorem}\label{thm:bernoulli}
With the above Markov chain, the transformation which maps a configuration $X(i,j)=a(i,j); -\infty\leq i\leq \infty,- \infty \leq j\leq \infty$ to 
$X(i,j)=a(i-1,j); -\infty \leq i\leq \infty,- \infty \leq j\leq \infty$
is Bernoulli.
\end{theorem}

Idea of proof: Reading backwards this is essentially a mixing Markov chain and using the proof of the analog of the renewal theorem, the proof that it is Bernoulli is identical to the proof that a mixing Markov chain on a finite state space is Bernoulli.

\begin{proof} Ornstein showed the following two results \cite{ornstein1} and \cite{ornstein2} resp. and result 3 is trivial directly from the definition of Bernoulli. 
\vskip .2cm
Ornstein Result 1:
Let $P_n$ be an increasing sequence of finite partitions which separate points and let $T$ be a measure preserving transformation. If the $T,P_n$ process is Bernoulli for all $n$ then $T$ is Bernoulli. 
\vskip .2cm
Ornstein Result 2:
Let $P$ be a finite partition and $T$ be a measure preserving transformation. If for every $\epsilon>0$ there is a coupling of two copies of the $T,P$ process so that they are independent in the past and the expected mean hamming distance of the $n$ future is less than \ep for sufficiently large $n$ then the $T,P$ process is Bernoulli. 
\vskip .2cm
Result 3:
The $T,P$ process is Bernoulli iff the $T^{-1},P$ process is Bernoulli. (this is trivial)
\vskip .2cm
From the above three lemmas,
\vskip 0 cm
Letting
\vskip .2cm
    $W(k)$ be the value of $X(i,j): -n\leq i\leq  n,  k-n\leq  j\leq k+n$ 
\vskip .2cm
if
\vskip .2cm
we can couple the $W$ process with another process referred to as the $Y$ process whose distribution is identical to the $W$ process so that
$W(k):k\geq 0$ is independent of $Y(k):k\geq 0$ but for sufficiently large $n$, 

$P( \#\{i:-n<i<0~$ and $~W(i) \neq Y(i)\}>\epsilon n)<\epsilon.$
\vskip .2cm
then
\vskip .2cm
The $X$ process is Bernoulli and the proof will be complete. 
\vskip .2cm
Extending the $W$ process to the $X$ process (on the whole two dimensional lattice) and extending the $Y$ process to a process to be called the $X'$ process with the same distribution as the $X$ process, it suffices to get a coupling of the entire $X(i,j)$ process with the entire $X'(i,j)$ process such that
\begin{align}
&1)~	\text{$X(i,j)~ j\geq -n$ is independent of $X'(i,j)~ j\geq -n$.} \notag \\
&2)~	\text{For $k$ sufficiently large $P(W(-k) \neq Y(-k))<\epsilon^3$} \notag
\end{align}
The only significance of $\epsilon^3$ is that it is smaller than $\epsilon^2$.

But we are already done because we can just couple the two processes independently for $j\geq -n$ and then couple them as in the proof of the analog to the renewal theorem for $j<-n$ and the proof of that theorem establishes (2). 
\end{proof}

\end{document}